\documentclass[10pt]{article}
\usepackage{lmodern}
\usepackage{amsmath}
\usepackage[T1]{fontenc}
\usepackage[utf8]{inputenc}
\usepackage{authblk}
\usepackage{amsfonts}
\usepackage{graphicx}
\usepackage{rotating}
\usepackage{amssymb}
\usepackage[english]{babel}
\usepackage{color}
\usepackage{amsthm}
\usepackage{graphicx}
\usepackage{tkz-euclide}
\tikzset{elegant/.style={smooth,thick,samples=50,cyan}}
\usepackage{color,tikz}
\usetikzlibrary{patterns}
\usetikzlibrary{decorations.pathreplacing,calc}
\usepackage{rotating}
\usepackage{mathrsfs}
\usepackage{makecell}
\usepackage{microtype}
\usepackage{enumerate}
\usepackage{mathscinet}
\usepackage{array}
\usepackage{multirow}
\usepackage{booktabs}
\usepackage{enumerate}
\usepackage[cal=boondoxo,bb=ams]{mathalfa}
\usepackage{hyperref}
\hypersetup{hidelinks}
\usepackage{titlesec}
\newtheorem{theorem}{Theorem}[section]
\newtheorem{prop}{Proposition}[section]
\newtheorem{lemma}{Lemma}[section]
\newtheorem{coro}{Corollary}[section]
\newtheorem{remark}{Remark}[section]
\newtheorem{exam}{Example}[section]

\newcommand{\ml}{\mathcal}
\newcommand{\mb}{\mathbb}

\DeclareMathOperator{\non}{nlin}
\DeclareMathOperator{\lin}{lin}
\DeclareMathOperator{\intt}{int}
\DeclareMathOperator{\extt}{ext}

\title{Critical exponent and sharp lifespan estimates for semilinear third-order evolution equations}
\author[1]{Wenhui Chen\thanks{Wenhui Chen (wenhui.chen.math@gmail.com)}}
\affil[1]{School of Mathematics and Information Science, Guangzhou University, 510006 Guangzhou, China}

\date{}
\begin{document}

\maketitle
\begin{abstract}
	\medskip
We study semilinear third-order (in time) evolution equations with  fractional Laplacian $(-\Delta)^{\sigma}$ and  power nonlinearity $|u|^p$, which was proposed by Bezerra-Carvalho-Santos \cite{Bezerra-Carvalho-Santos=2022} recently. In this manuscript, we obtain a new critical exponent $p=p_{\mathrm{crit}}(n,\sigma):=1+\frac{6\sigma}{\max\{3n-4\sigma,0\}}$ for $n\leqslant\frac{10}{3}\sigma$. Precisely, the global (in time) existence of small data Sobolev solutions is proved for the supercritical case $p>p_{\mathrm{crit}}(n,\sigma)$, and weak solutions blow up in finite time even for small data if $1<p\leqslant p_{\mathrm{crit}}(n,\sigma)$. Furthermore, to more accurately describe the blow-up time, we derive new and sharp upper bound as well as lower bound estimates for the lifespan in the subcritical case and the critical case.\\
	
	
	\noindent\textbf{AMS Classification (2020)} 35L76, 35L30, 35B33, 35A01, 35B44 
\end{abstract}
\fontsize{12}{15}
\selectfont
\section{Introduction}
$\ \ \ \ $In the present paper, we consider the following Cauchy problem for third-order (in time) evolution equations (proposed by Bezerra-Carvalho-Santos \cite{Bezerra-Carvalho-Santos=2022} recently):
\begin{align}\label{Eq-Third-PDE}
\begin{cases}
u_{ttt}+\ml{A}u+\eta\ml{A}^{\frac{1}{3}}u_{tt}+\eta\ml{A}^{\frac{2}{3}}u_t=f(u),&x\in\mb{R}^n,\ t>0,\\
u(0,x)=\epsilon u_0(x),\ u_t(0,x)=\epsilon u_1(x),\ u_{tt}(0,x)=\epsilon u_2(x),&x\in\mb{R}^n,
\end{cases}
\end{align}
with a positive constant $\eta$ and power nonlinearity $f(u):=|u|^p$ carrying $p>1$, where the operator $\ml{A}^{\alpha}$ with
 fractional Laplacian $\ml{A}:=(-\Delta)^{\sigma}$ carrying $\sigma\in(0,\infty)$ is defined via
\begin{align*}
	\ml{A}^{\alpha}u=(-\Delta)^{\sigma\alpha}u:=\ml{F}^{-1}_{\xi\to x}\big(|\xi|^{2\sigma\alpha}\,\ml{F}_{x\to\xi}(u)(t,\xi)\big)\ \ \mbox{with}\ \ \alpha\in[0,1].
\end{align*}
In the above model \eqref{Eq-Third-PDE}, $\epsilon>0$ is a positive parameter describing the size of initial data. Note that the settings $\ml{A}=(-\Delta)^{\sigma}$ and $f(u)=|u|^p$ will be used throughout this paper unless otherwise specified. Our first purpose is to determine a critical exponent for the nonlinear third-order (in time) evolution equations \eqref{Eq-Third-PDE} with fractional Laplacian $\ml{A}=(-\Delta)^{\sigma}$ carrying $\sigma\in(0,\infty)$. Here, the critical exponent means the threshold condition on the exponent $p$ for global (in time) Sobolev solutions and blow-up of local (in time) weak solutions with small data. To be specific, under additional $L^1$ integrable assumptions for initial data, the new critical exponent, which will raise for the nonlinear Cauchy problem \eqref{Eq-Third-PDE} with $\eta\in(1,\infty)$, is
\begin{align}\label{Crit} 
	p=p_{\mathrm{crit}}(n,\sigma):=1+\frac{6\sigma}{(3n-4\sigma)_+}\ \ \mbox{for}\ \ 1\leqslant n\leqslant\frac{10}{3}\sigma.
\end{align}
A further purpose of this work is to derive sharp lifespan estimates for local (in time) solutions to the nonlinear Cauchy problem \eqref{Eq-Third-PDE} with $\eta\in(1,\infty)$, where the lifespan $T_{\epsilon}$ of a solution is defined by
\begin{align}\label{Lifespan_Defn}
	T_{\epsilon}&:=\sup\big\{T>0:\ \mbox{there exists a unique local (in time) solution }u\mbox{ to  the nonlinear}\notag\\ &\qquad\qquad\qquad\quad\ \mbox{ Cauchy problem \eqref{Eq-Third-PDE} on }[0,T)\mbox{ with a fixed parameter }\epsilon>0\big\}.
\end{align}
We will rigorously demonstrate  the sharpness of new lifespan estimates
\begin{align*}
T_{\epsilon}\begin{cases}
\simeq C\epsilon^{-\frac{2\sigma}{6\sigma p'-(3n+2\sigma)}}&\mbox{if}\ \  p<p_{\mathrm{crit}}(n,\sigma),\\
\simeq \exp(C\epsilon^{-(p-1)})&\mbox{if}\ \ p=p_{\mathrm{crit}}(n,\sigma),\\
=\infty&\mbox{if}\ \ p>p_{\mathrm{crit}}(n,\sigma),
\end{cases}
\end{align*}
in which $C$ is an independent of $\epsilon$, positive constant. In the last power, we denote by $p'$ H\"older's conjugate of $p>1$ such that $p'=\frac{p}{p-1}$. It is worth mentioning that the corresponding linearized Cauchy problem to \eqref{Eq-Third-PDE} with Sobolev data is local (in time) ill-posedness if $\eta\in(0,1)$ which will be proved in our manuscript.

Let us now recall several background related to our model \eqref{Eq-Third-PDE}. First of all, the authors of \cite{Fattorini=1983,Bezerra-Santos=2020} considered approximations of third-order (in time) linear evolution equations \eqref{Eq-Third-PDE} with $\eta=0$ and $f(u)\equiv0$, whose initial data belongs to some separable Hilbert spaces. The work \cite{Bezerra-Carvalho-Santos=2022} characterized the partial scale of the fractional power of order spaces associated with these operators. Later, the recent paper \cite{Bezerra-Carvalho-Santos=2022} found the threshold $\eta=1$ for the stability of the linearized model to \eqref{Eq-Third-PDE} with $f(u)\equiv0$ by analyzing the eigenvalues of  its resolvent operator. Furthermore, the authors of \cite{Bezerra-Carvalho-Santos=2022} presented the local (in time) solvability for the boundary value problem to \eqref{Eq-Third-PDE} with $\sigma=1$, $\eta\in(1,\infty)$ as well as continuously differentiable nonlinearity fulfilling the growth condition $|f'(s)|\lesssim 1+|s|^{p-1}$ carrying $1<p<\frac{3n+4}{3n-8}$ when $n\geqslant 3$. Nevertheless,  to the best knowledge of author,  global (in time) existence or blow-up of solutions to the nonlinear model \eqref{Eq-Third-PDE} are still not clear due to the high-order time-derivative and the nonlocal operator $\ml{A}$. 

More importantly, the critical exponent and sharp lifespan estimates are cutting-edge topics in semilinear evolution equations with power nonlinearity $f(u)=|u|^p$, for example, the well-known Fujita exponent $p_{\mathrm{Fuj}}(n):=1+\frac{2}{n}$ for reaction-diffusion equations \cite{Fujita-1966,Esco-Herr-1991,Esco-Levi-1995,Fuji-Iked-Waka-2020}, classical damped wave equations \cite{Li-Zhou-1995,Todo-Yord-2001,Zhang-2001,Iked-Ogaw-2016,Lai-Zhou=2019,Ebert-Girardi-Reissig=2020} and some wave equations with effective damping \cite{Wakasugi=2012,D'Abbicco-Lucente-Reissig=2013,D'Abbicco-Ebert=2017,Nasc-Palmieri-Reissig=2017} and references therein. These researches motivate us to study the critical exponent and sharp lifespan estimates to the semilinear Cauchy problem \eqref{Eq-Third-PDE} with $f(u)=|u|^p$. Note that the critical exponent $p_{\mathrm{crit}}(n,\sigma)$ proposed in \eqref{Crit} is not a generalization of the Fujita exponent $p_{\mathrm{Fuj}}(n)$. Lastly,  we also refer interested readers to another kind of important third-order (in time) PDEs, which is the so-called Jordan-Moore-Gibson-Thompson equation \cite{Kaltenbacher-Lasiecka-Pos-2012,Racke-Said-2020,B-L-2020,Kaltenbacher-Nikolic-2021,Said-2021-Bes} in recent studies of nonlinear acoustics models in thermoviscous flows.

To explore the critical exponent for the semilinear Cauchy problem \eqref{Eq-Third-PDE}, we firstly study some stabilities, sharp estimates and asymptotic profiles of solutions to the linearized Cauchy problem \eqref{Eq-Linear-Third-PDE} by applying Fourier analysis in Section \ref{Sec-linear}, where  new thresholds $\eta=1$ and $\eta=3$ will be stated (see Figure \ref{imgg}). Then, in Section \ref{Sec-GESDS} by constructing suitable time-weighted Sobolev spaces, we demonstrate the global (in time) existence of small data solution when $p>p_{\mathrm{crit}}(n,\sigma)$. For another, in the case $1<p\leqslant p_{\mathrm{crit}}(n,\sigma)$, every non-trivial weak solution blows up in finite time under a  condition for initial data, which will be proved in Section \ref{Sec-Blow-up}. They infer the critical exponent \eqref{Crit}.

To derive sharp lifespan estimates for the semilinear Cauchy problem \eqref{Eq-Third-PDE}, we separate our discussion into several cases. The upper bound estimates when $1<p< p_{\mathrm{crit}}(n,\sigma)$ and the sharp lower bound estimates when $1<p\leqslant p_{\mathrm{crit}}(n,\sigma)$ can be obtained by following the proofs in Sections \ref{Sec-GESDS} and \ref{Sec-Blow-up}, separately. With the aim of completing the sharp upper bound estimates in the critical case $p=p_{\mathrm{crit}}(n,\sigma)$, motivated by \cite{Ikeda-Sobajima=2019,Ebert-Girardi-Reissig=2020} for the semilinear classical damped waves, we introduce the test function \eqref{Su01} with suitable scaling due to fractional Laplacian $\ml{A}^{\alpha}=(-\Delta)^{\alpha\sigma}$. These results in Section \ref{Sec-Lifespan} conclude the sharpness of lifespan estimates immediately.

\medskip
\noindent\textbf{Notations:}
\medskip

$\ \  $Firstly, $c$ and $C$ denote some positive constants, which may be changed
from line to line. 
 $B_R$ stands for a ball around the origin with radius $R$ in $\mb{R}^n$.
 We write $f\lesssim g$ if there exists a positive
constant $C$ such that $f\leqslant Cg$. 
Moreover, $f\simeq g$ means that $f\lesssim g$ and $g\lesssim f$, simultaneously. 

 The Japanese bracket is denoted by $\langle x\rangle:=\sqrt{1+|x|^2}$. 
Additionally, we write $(x)_+:=\max\{x,0\}$, and $\frac{1}{(x)_+}:=\infty$ if $x\leqslant0$.
We take $[x]$ to be the integer part of $x\in\mb{R}_+$.

 We define two zones in the Fourier space ($0<\varepsilon_0\ll 1$):
\begin{align*}
	\ml{Z}_{\intt}(\varepsilon_0):=\{|\xi|\leqslant\varepsilon_0\}\ \ \mbox{and}\ \ 
	\ml{Z}_{\extt}(\varepsilon_0):=\{|\xi|\geqslant \varepsilon_0\}.
\end{align*}
The corresponding $\ml{C}^{\infty}$ cutoff functions $\chi_{\intt}(\xi)$ and $\chi_{\extt}(\xi)$ equipping their supports in the zones $\ml{Z}_{\intt}(\varepsilon_0)$ and $\ml{Z}_{\extt}(\varepsilon_0/2)$, respectively, so that $\chi_{\extt}(\xi)=1-\chi_{\intt}(\xi)$ for all $\xi \in \mb{R}^n$. 

A given function $g=g(x)$ belongs to the Gevrey-Sobolev space $G^{m,s}$ if there exists parameters $m\in\mb{R}_+$ and $s\in\mb{R}$ such that  (see, for example, \cite{Rodino=1993})
\begin{align*}
\exp\left(c\langle \xi\rangle^{\frac{1}{m}}\right)\langle\xi\rangle^s\,\ml{F}(g)\in L^2
\end{align*}
with  a positive constant $c$. Finally, we introduce the Sobolev space with additional $L^1$ integrability to be the initial data space as follows:
\begin{align}\label{Data-space-D}
\ml{B}_{\sigma}:=\big(H^{\frac{4}{3}\sigma}\cap L^1\big)\times \big(H^{\frac{2}{3}\sigma}\cap L^1\big)\times \big(L^2\cap L^1\big)
\end{align}
with $\sigma\in(0,\infty)$ in general.

\section{Main results and their explanations}
\subsection{Results and discussions for the critical exponent}
$\ \ \ \ $Let us state the global (in time) existence result in the supercritical case $p>p_{\mathrm{crit}}(n,\sigma)$ introduced by \eqref{Crit} with the initial data space \eqref{Data-space-D}.
\begin{theorem}\label{Thm-GESDS}
Let $\frac{4}{3}\sigma<n\leqslant\frac{10}{3}\sigma$ with $\sigma\in(0,\infty)$. Let us consider the semilinear Cauchy problem \eqref{Eq-Third-PDE} with $\eta\in(1,\infty)$ and  $(u_0,u_1,u_2)\in\ml{B}_{\sigma}$. Let the exponent $p$ fulfill
	\begin{align}\label{H1}
	p_{\mathrm{crit}}(n,\sigma)<p\leqslant \frac{3n}{(3n-8\sigma)_+}.
	\end{align}
	Then, there exists a constant $C>0$ and a parameter $0<\epsilon\ll 1$ such that for all $\|(u_0,u_1,u_2)\|_{\ml{B}_{\sigma}}\leqslant C$, there is a uniquely determined Sobolev solution
	\begin{align}\label{Solution-u}
	u\in\ml{C}\big([0,\infty),H^{\frac{4}{3}\sigma}\big).
	\end{align}
	Consequently, the lifespan of solutions is given by $T_{\epsilon}=\infty$. Furthermore, the solution satisfies the following estimates:
	\begin{align*}
	\|u(t,\cdot)\|_{L^2}&\lesssim
	\begin{cases}
	\epsilon\ln (\mathrm{e}+t)\|(u_0,u_1,u_2)\|_{\ml{B}_{\sigma}}&\mbox{if}\ \ n=\frac{8}{3}\sigma\ \mbox{when}\ \eta\in(3,\infty),\\
	\epsilon (1+t)^{-\frac{3n-8\sigma}{4\sigma}}\|(u_0,u_1,u_2)\|_{\ml{B}_{\sigma}}&\mbox{otherwise,}
	\end{cases}
		\end{align*}
		and
	\begin{align*}	
		\|u(t,\cdot)\|_{\dot{H}^{\frac{4}{3}\sigma}}&\lesssim\epsilon (1+t)^{-\frac{3n}{4\sigma}}\|(u_0,u_1,u_2)\|_{\ml{B}_{\sigma}}.
	\end{align*}
\end{theorem}
\begin{remark}
The restriction $n\leqslant\frac{10}{3}\sigma$ in Theorem \ref{Thm-GESDS} contributes to the lower bound of the exponent $p$ such that $p>p_{\mathrm{crit}}(n,\sigma)$ due to our purpose of the critical exponent. Note that \eqref{H1} is non-empty. Actually, we also can prove global (in time) existence result with other ranges of the parameters $\sigma$, $n$, $p$, e.g. the global (in time) Sobolev solution \eqref{Solution-u} uniquely exists when
\begin{align*}
	p_{\mathrm{crit}}(n,\sigma)<p\leqslant\frac{3n}{(3n-8\sigma)_+}\ \ \mbox{and}\ \ p\geqslant 2,
\end{align*}
in the case $\frac{4}{3}\sigma<n<4\sigma$ if $\eta\in(1,3]$, in which  $p\geqslant 2$ originates from the application of the fractional Gagliardo-Nirenberg inequality (see Subsection \ref{Subsec-GESDS}). Another example for $4\sigma\leqslant n\leqslant\frac{16}{3}\sigma$ will be stated in Remark \ref{Rem-Scrop-02}. However, it is beyond the scope of our manuscript.
\end{remark}
\begin{remark}
	By lengthy but straightforward derivations, under the same condition of the exponent $p$ as \eqref{H1}, the regularity of global (in time) solution \eqref{Solution-u} can be improved by
	\begin{align*}
		u\in\ml{C}\big([0,\infty),H^{\frac{4}{3}\sigma}\big)\cap \ml{C}^1\big([0,\infty),H^{\frac{2}{3}\sigma}\big)\cap\ml{C}^2\big([0,\infty),L^2\big)
	\end{align*}
	with some estimates for time-derivatives of solutions. The further explanation will be shown later in Remark \ref{Rem-Further_regularity}.
\end{remark}
\begin{remark}
Some derived growth/decay rates in Theorem \ref{Thm-GESDS} coincide with those of the corresponding linearized Cauchy problem  \eqref{Eq-Linear-Third-PDE} when $\eta\in(1,\infty)$ in Propositions \ref{Prop-(1,3)}-\ref{Prop-(3)}, which verifies the phenomenon of no loss of growth/decay. Note that the solution itself decays polynomially only when $\frac{8}{3}\sigma<n<\frac{10}{3}\sigma$  in Theorem \ref{Thm-GESDS}. Particularly, Proposition \ref{Prop-sharpness-linear} rigorously demonstrates sharpness of the $(L^2\cap L^1)-L^2$ type estimates for the linearized problem \eqref{Eq-Linear-Third-PDE} so that we conjecture the sharpness of some derived estimates in Theorem \ref{Thm-GESDS}. But, there is a logarithmic loss technically for $n=\frac{8}{3}\sigma$ when $\eta\in(3,\infty)$.
\end{remark}

We now turn our mind to a blow-up result in the subcritical case $p<p_{\mathrm{crit}}(n,\sigma)$ as well as the critical case $p=p_{\mathrm{crit}}(n,\sigma)$ under a sign condition of the third initial data.

\begin{theorem}\label{Thm-Blow-up}
 Let us consider the semilinear Cauchy problem \eqref{Eq-Third-PDE} with $\eta\in(1,\infty)$, $\sigma\in(0,\infty)$, and $u_0,u_1,u_2\in L^1$ such that
\begin{align}\label{Sign_Assumption}
\int_{\mb{R}^n}u_2(x)\mathrm{d}x>0.
\end{align}
Let the exponent $p$ fulfill $1<p\leqslant p_{\mathrm{crit}}(n,\sigma)$. Then, according to the definition \eqref{Weak_Solution}, every non-trivial local (in time) weak solution for the semilinear model \eqref{Eq-Third-PDE}  blows up in finite time.
\end{theorem}
\begin{remark}
The blow-up result still holds when
\begin{align*}
\int_{\mb{R}^n}\left(\eta \ml{A}^{\frac{2}{3}}u_0(x)+\eta\ml{A}^{\frac{1}{3}} u_1(x)+u_2(x)\right)\mathrm{d}x>0
\end{align*}
instead of the single sign condition \eqref{Sign_Assumption}, whereas we suppose that $\ml{A}^{\frac{2}{3}}u_0,\ml{A}^{\frac{1}{3}} u_1,u_2\in L^1$ with $\ml{A}=(-\Delta)^{\sigma}$. Namely, another sign condition of getting the blow-up result with $1<p\leqslant p_{\mathrm{crit}}(n,\sigma)$ even for the failure of \eqref{Sign_Assumption}. The proof is just a slight modification to the left-hand side of \eqref{Blow-up-01}.
\end{remark}
\begin{remark}
Recalling the notation \eqref{Crit}, the condition $1<p\leqslant p_{\mathrm{crit}}(n,\sigma)$ in Theorem \ref{Thm-Blow-up} includes the special case that every non-trivial weak solution blows up with any $p>1$ when $n\leqslant\frac{4}{3}\sigma$ and $\sigma\in(0,\infty)$. This is the reason for the general restriction $n>\frac{4}{3}\sigma$ in the global (in time) existence result from Theorem \ref{Thm-GESDS}.
\end{remark}

To end this subsection, we analyze the new critical exponent to the Cauchy problem for semilinear third-order (in time) evolution equations \eqref{Eq-Third-PDE}. According to Theorems \ref{Thm-GESDS} and \ref{Thm-Blow-up}, we immediately conclude the critical exponent $p=p_{\mathrm{crit}}(n,\sigma)$ given by \eqref{Crit} when $n\leqslant\frac{10}{3}\sigma$ and $\eta\in(1,\infty)$. The critical exponent for the higher dimensional case, i.e. $n> \frac{10}{3}\sigma$, may be obtained by using more general $L^r-L^q$ estimates, $1\leqslant r\leqslant q\leqslant \infty$,  but this task is beyond the scope of this paper.

\begin{exam}
Let us consider the semilinear problem \eqref{Eq-Third-PDE} with $\sigma=3$ and $\eta\in(1,\infty)$, namely,
\begin{align}\label{Example}
	\begin{cases}
		u_{ttt}-\Delta^3 u-\eta\Delta u_{tt}+\eta\Delta^2 u_t=|u|^p,&x\in\mb{R}^n,\ t>0,\\
		u(0,x)=\epsilon u_0(x),\ u_t(0,x)=\epsilon u_1(x),\ u_{tt}(0,x)=\epsilon u_2(x),&x\in\mb{R}^n,
	\end{cases}
\end{align}
in which the equation can be understood by the heat operator acting on the semilinear structurally damped plate model \cite{Pham-Kainane-Reissig=2015,D'Abbicco-Ebert=2017} as follows:
\begin{align*}
(\partial_t-\Delta)\big(u_{tt}+\Delta^2 u-(\eta-1)\Delta u_{t}\big)=|u|^p.
\end{align*}
Combining with Theorems \ref{Thm-GESDS} and \ref{Thm-Blow-up}, we assume the small data with $0<\epsilon\ll 1$ satisfying
\begin{align}\label{H2}
	(u_0,u_1,u_2)\in (H^{4}\cap L^1)\times (H^{2}\cap L^1)\times (L^2\cap L^1) \ \ \mbox{and}\ \ \int_{\mb{R}^n}u_2(x)\mathrm{d}x>0.
\end{align}
Then, the critical exponent for \eqref{Example} is given by
\begin{align*}
	p=p_{\mathrm{crit}}(n,3)=1+\frac{6}{(n-4)_+}
\end{align*}
when $n=1,\dots,10$ for all $\eta\in(1,\infty)$.
\end{exam}

\subsection{Results and discussions for sharp lifespan estimates}
$\ \ \ \ $As we explained in Theorem \ref{Thm-Blow-up}, under the vital condition $1<p\leqslant p_{\mathrm{crit}}(n,\sigma)$, every non-trivial
local (in time) solution blows up in finite time, which motivates us to provide more detailed information of the lifespan $T_{\epsilon}$ defined in \eqref{Lifespan_Defn}.

Before investigating lower bound estimates for the lifespan $T_{\epsilon}$, let us define mild solutions to the semilinear Cauchy problem \eqref{Eq-Third-PDE} with $T>0$ for $u\in\ml{C}([0,T],H^{\frac{4}{3}\sigma})$ as solutions of the following operator equality:
\begin{align}\label{Mild-Solution}
u(t,x)=\epsilon\sum\limits_{j=0,1,2}E_j(t,x)\ast_{(x)}u_j(x)+\int_0^tE_2(t-\tau,x)\ast_{(x)}|u(\tau,x)|^p\mathrm{d}\tau
\end{align}
for $t\in[0,T]$. In the above equation, $E_j(t,x)$ with $j=0,1,2$ are the fundamental solutions to the corresponding linearized Cauchy problem \eqref{Eq-Linear-Third-PDE} with initial data $v_j(x)=\delta_0$ and $v_k(x)=0$ carrying $k\neq j$. Here, $\delta_0$ is the Dirac distribution at $x=0$. Let us introduce by $T_{\epsilon,\mathrm{m}}$ the lifespan of a mild solution $u=u(t,x)$. Then, lower bound estimates of $T_{\epsilon,\mathrm{m}}$ are given by the next theorem.
\begin{theorem}\label{Thm-Lifespan-lower}
Let $\frac{4}{3}
\sigma<n\leqslant\frac{10}{3}\sigma$ if $\eta\in(1,3)\cup(3,\infty)$, and $n\leqslant\frac{10}{3}\sigma$ if $\eta=3$ with $\sigma\in(0,\infty)$. Let us consider the semilinear Cauchy problem \eqref{Eq-Third-PDE} with $\eta\in(1,\infty)$ and  $(u_0,u_1,u_2)\in\ml{B}_{\sigma}$. Let the exponent $p$ fulfill $2\leqslant p\leqslant p_{\mathrm{crit}}(n,\sigma)$. Then, there exists a constant $\epsilon_1$ such that for any $\epsilon\in(0,\epsilon_1]$, the lifespan $T_{\epsilon,\mathrm{m}}$ of mild solutions from the definition \eqref{Mild-Solution} satisfies
\begin{align}\label{Lifespan+Upper+Bound}
T_{\epsilon,\mathrm{m}}\geqslant\begin{cases}
C \epsilon^{-\frac{2\sigma}{6\sigma p'-(3n+2\sigma)}}&\mbox{if}\ \  p<p_{\mathrm{crit}}(n,\sigma),\\
 \exp(C\epsilon^{-(p-1)})&\mbox{if}\ \ p=p_{\mathrm{crit}}(n,\sigma),
\end{cases}
\end{align}
where $C$ is a positive constant independent of $\epsilon$ and depends  on $p,n,\sigma,\eta$ as well as $\|(u_0,u_1,u_2)\|_{\ml{B}_{\sigma}}$.
\end{theorem}
\begin{remark}
The restriction $n\leqslant\frac{10}{3}\sigma$ contributes to the non-empty set $2\leqslant p\leqslant p_{\mathrm{crit}}(n,\sigma)$ because the lifespan $T_{\epsilon}=\infty$ when $p>p_{\mathrm{crit}}(n,\sigma)$ proposed in Theorem \ref{Thm-GESDS}. Here, the lower bound $p\geqslant 2$, again, comes from the application of the fractional Gagliardo-Nirenberg inequality, which can be improved by using more general $L^r-L^q$ estimates carrying $1\leqslant r\leqslant q\leqslant \infty$.
\end{remark}

To guarantee sharpness of the derived lifespan estimates \eqref{Lifespan+Upper+Bound}, we have to estimate it from the above side. Therefore, we turn to upper bound estimates for the lifespan $T_{\epsilon,\mathrm{w}}$ of a weak solution in the subsequent theorem.
\begin{theorem}\label{Thm-Lifespan-upper}
	Let us consider the semilinear Cauchy problem  \eqref{Eq-Third-PDE} with $\eta\in(1,\infty)$ and initial data satisfying the sign condition \eqref{Sign_Assumption}.
	\begin{itemize}
		\item In the subcritical case $1<p<p_{\mathrm{crit}}(n,\sigma)$,  assuming $u_0,u_1,u_2\in L^1$ and $\sigma\in(0,\infty)$, the lifespan $T_{\epsilon,\mathrm{w}}$ of weak solutions  satisfies
		\begin{align}\label{Upper-1}
		T_{\epsilon,\mathrm{w}}\leqslant C \epsilon^{-\frac{2\sigma}{6\sigma p'-(3n+2\sigma)}}.
		\end{align}
		\item In the critical case $p=p_{\mathrm{crit}}(n,\sigma)$, assuming $u_0,u_1,u_2\in \ml{C}_0^{\infty}$ and $\sigma\in3\mb{N}$, the lifespan  $T_{\epsilon,\mathrm{w}}$ of weak solutions satisfies
		\begin{align}\label{Upper-2}
		T_{\epsilon,\mathrm{w}}\leqslant \exp(C\epsilon^{-(p-1)}).
		\end{align}
	\end{itemize}
Here, $C$ is a positive constant independent of $\epsilon$ and depends on $p,n,\sigma,\eta$ as well as $\|u_2\|_{L^1}$.
\end{theorem}
\begin{remark}
In the critical case, the technical assumption $\sigma\in3\mb{N}$ comes from the treatment of the operator $\ml{A}^{\frac{1}{3}}=(-\Delta)^{\frac{1}{3}\sigma}$. The question of sharp upper bound estimates \eqref{Upper-2} when $p=p_{\mathrm{crit}}(n,\sigma)$ for a general parameter $\sigma\in(0,\infty)$ is still open.
\end{remark}

By the standard density argument (see, for example, \cite[Proposition 3.1]{Ikeda-Wakasugi=2013}), provided that $u$ is a mild solution to the semilinear Cauchy problem \eqref{Eq-Third-PDE}, then $u$ is also a weak solution to it. Therefore, the relation $T_{\epsilon,\mathrm{m}}\leqslant T_{\epsilon}\leqslant T_{\epsilon,\mathrm{w}}$ is valid.
Let us summarize the lower bound estimates \eqref{Lifespan+Upper+Bound}, and upper bound estimates \eqref{Upper-1}-\eqref{Upper-2}. We may claim that the sharp lifespan estimates to $T_{\epsilon}$ for local (in time) solutions to the semilinear Cauchy problem \eqref{Eq-Third-PDE} with $p\geqslant 2$ are given by
\begin{align*}
T_{\epsilon}\simeq\begin{cases}
 C\epsilon^{-\frac{2\sigma}{6\sigma p'-(3n+2\sigma)}}&\mbox{if}\ \  p<p_{\mathrm{crit}}(n,\sigma),\\
 \exp\big(C\epsilon^{-\frac{6\sigma}{(3n-4\sigma)_+}}\big)&\mbox{if}\ \ p=p_{\mathrm{crit}}(n,\sigma),\\
\end{cases}
\end{align*}
for some parameters $n,\sigma$, with a positive constant $C$ independent of $\epsilon$.
\begin{exam}
Let us consider the semilinear Cauchy problem \eqref{Example}. Combining with Theorems \ref{Thm-Lifespan-lower} and \ref{Thm-Lifespan-upper}, we assume the small data with $0<\epsilon\ll 1$ satisfying \eqref{H2},
and additionally $\ml{C}_0^{\infty}$ regularities with supports for initial data if $p=p_{\mathrm{crit}}(n,3)=1+\frac{6}{(n-4)_+}$. Then, the sharp lifespan estimates for \eqref{Example} with $p\geqslant 2$ are given by
\begin{align*}
T_{\epsilon}\simeq\begin{cases}
C\epsilon^{-\frac{2}{6 p'-(n+2)}}&\mbox{if}\ \  p<1+\frac{6}{(n-4)_+},\\
\exp\big(C\epsilon^{-\frac{6}{(n-4)_+}}\big)&\mbox{if}\ \ p=1+\frac{6}{(n-4)_+},\\
\end{cases}
\end{align*}
when $n=5,\dots,10$ if $\eta\in(1,3)\cup(3,\infty)$, and $n=1,\dots,10$ when $\eta=3$.
\end{exam}

\section{Linearized third-order (in time) evolution equations}\label{Sec-linear}
$\ \ \ \ $As a preparation to study the semilinear Cauchy problem \eqref{Eq-Third-PDE} and investigate some influence of the parameter $\eta$, in this section, we consider the corresponding linearized third-order (in time) evolution equations with vanishing right-hand side, namely,
\begin{align}\label{Eq-Linear-Third-PDE}
	\begin{cases}
		v_{ttt}+\ml{A}v+\eta\ml{A}^{\frac{1}{3}}v_{tt}+\eta\ml{A}^{\frac{2}{3}}v_t=0,&x\in\mb{R}^n,\ t>0,\\
		v(0,x)=v_0(x),\ v_t(0,x)= v_1(x),\ v_{tt}(0,x)=v_2(x),&x\in\mb{R}^n,
	\end{cases}
\end{align}
where $\ml{A}=(-\Delta)^{\sigma}$ with $\sigma\in(0,\infty)$ and $\eta\in(0,\infty)$. In particular, we will introduce two thresholds: $\eta=1$ for the Sobolev stability, and $\eta=3$ for  various asymptotic profiles of solutions. One may see the detailed explanations in Remark \ref{Rem-eta=1}, Remark \ref{Rem-eta=3} as well as Figure \ref{imgg} later.

\subsection{Stability analysis of Sobolev solutions}
$\ \ \ \ $Let us apply the partial Fourier transform with respect to the spatial variables to the linearized Cauchy problem \eqref{Eq-Linear-Third-PDE} to derive
\begin{align}\label{Eq-Fourier-Third-PDE}
	\begin{cases}
		\widehat{v}_{ttt}+|\xi|^{2\sigma}\widehat{v}+\eta|\xi|^{\frac{2}{3}\sigma}\widehat{v}_{tt}+\eta|\xi|^{\frac{4}{3}\sigma}\widehat{v}_t=0,&\xi\in\mb{R}^n,\ t>0,\\
		\widehat{v}(0,\xi)=\widehat{v}_0(\xi),\ \widehat{v}_t(0,\xi)= \widehat{v}_1(\xi),\ \widehat{v}_{tt}(0,\xi)=\widehat{v}_2(\xi),&\xi\in\mb{R}^n,
	\end{cases}
\end{align}
whose characteristic equation
\begin{align*}
	(\lambda+|\xi|^{\frac{2}{3}\sigma})\big(\lambda^2+(\eta-1)|\xi|^{\frac{2}{3}\sigma}\lambda+|\xi|^{\frac{4}{3}\sigma}\big)=0
\end{align*}
has the $|\xi|$-dependent roots
\begin{align}\label{Characteristic_Roots}
\lambda_1=-|\xi|^{\frac{2}{3}\sigma}\ \ \mbox{and}\ \ \lambda_{2,3}=\frac{1}{2}\left(1-\eta\pm\sqrt{\eta^2-2\eta-3}\,\right)|\xi|^{\frac{2}{3}\sigma}.
\end{align}

Therewith, we may demonstrate the ill-/well-posedness results of Sobolev solutions to the linearized Cauchy problem \eqref{Eq-Linear-Third-PDE} under different size of $\eta$ and regularities of initial data.
\begin{prop}\label{Prop-Stable} Let us consider the linearized Cauchy problem \eqref{Eq-Linear-Third-PDE} with $\eta\in(0,\infty)$ as well as $\sigma\in(0,\infty)$.
	\begin{enumerate}[(1)]
\item Assuming the Gevrey-Sobolev regularities for initial data $v_j\in G^{\frac{3}{2\sigma},s+\frac{2(2-j)}{3}\sigma}$ with $j=0,1,2$, there exists a unique Sobolev solution 
\begin{align}\label{Sobolev-solution}
	v\in\ml{C}\big([0,\infty),H^{s+\frac{4}{3}\sigma}\big)\cap \ml{C}^1\big([0,\infty),H^{s+\frac{2}{3}\sigma}\big)\cap \ml{C}^2\big([0,\infty),H^{s}\big)
\end{align}
to the Cauchy problem \eqref{Eq-Linear-Third-PDE} with $\eta\in(0,1)$ for any $s\in[0,\infty)$. However, by assuming any Sobolev regularities for initial data, the Cauchy problem \eqref{Eq-Linear-Third-PDE} with $\eta\in(0,1)$ is ill-posed.
\item Assuming the Sobolev regularities for initial data $v_j\in H^{s+\frac{2(2-j)}{3}\sigma}$ with $j=0,1,2$, there exists a unique Sobolev solution \eqref{Sobolev-solution} to the Cauchy problem \eqref{Eq-Linear-Third-PDE} with $\eta=1$ for any $s\in[0,\infty)$.
\item  Assuming the Sobolev regularities for initial data $v_j\in H^{s+\frac{2(2-j)}{3}\sigma}$ with $j=0,1,2$, there exists a unique Gevrey-Sobolev solution
\begin{align*}
	v\in\ml{C}\big([0,\infty),G^{\frac{3}{2\sigma},s+\frac{4}{3}\sigma}\big)\cap \ml{C}^1\big([0,\infty),G^{\frac{3}{2\sigma},s+\frac{2}{3}\sigma}\big)\cap \ml{C}^2\big([0,\infty),G^{\frac{3}{2\sigma},s}\big)
\end{align*}
to the Cauchy problem \eqref{Eq-Linear-Third-PDE} with $\eta\in(1,\infty)$ for any $s\in[0,\infty)$.
	\end{enumerate}
\end{prop}
\begin{proof}
According to the size of $\eta$ and the value of characteristic roots $\lambda_{2,3}$, we will divide our discussion into several cases. To begin with the proof, let us derive  representations of solution in the Fourier space with different size of $\eta$.\\
\underline{\bf Representation of solution when $\eta\in(0,3)$:} Because of $\eta^2-2\eta-3<0$ when $\eta\in(0,3)$, by taking $\lambda_{2,3}=\mu_{\mathrm{R}}\pm i\mu_{\mathrm{I}}$ endowed with
\begin{align*}
	\mu_{\mathrm{R}}:=\frac{1}{2}(1-\eta)|\xi|^{\frac{2}{3}\sigma}\begin{cases}
	>0&\mbox{if}\ \ \eta\in(0,1),\\
	=0&\mbox{if}\ \ \eta=1,\\
	<0&\mbox{if}\ \ \eta\in(1,3),
	\end{cases}\ \   \mbox{and}\ \ \mu_{\mathrm{I}}:=\frac{1}{2}\sqrt{3+2\eta-\eta^2}\,|\xi|^{\frac{2}{3}\sigma},
\end{align*}
the solution $\widehat{v}=\widehat{v}(t,\xi)$ to the initial value problem \eqref{Eq-Fourier-Third-PDE} can be expressed by
\begin{align}
	\widehat{v}&=\frac{-(\mu_{\mathrm{I}}^2+\mu_{\mathrm{R}}^2)\widehat{v}_0+2\mu_{\mathrm{R}}\widehat{v}_1-\widehat{v}_2}{2\mu_{\mathrm{R}}\lambda_1-\mu_{\mathrm{I}}^2-\mu_{\mathrm{R}}^2-\lambda_1^2}\mathrm{e}^{\lambda_1t}+\frac{(2\mu_{\mathrm{R}}\lambda_1-\lambda_1^2)\widehat{v}_0-2\mu_{\mathrm{R}}\widehat{v}_1+\widehat{v}_2}{2\mu_{\mathrm{R}}\lambda_1-\mu_{\mathrm{I}}^2-\mu_{\mathrm{R}}^2-\lambda_1^2}\cos(\mu_{\mathrm{I}}t)\,\mathrm{e}^{\mu_{\mathrm{R}}t}\notag\\
	&\quad+\frac{\lambda_1(\mu_{\mathrm{R}}\lambda_1+\mu_{\mathrm{I}}^2-\mu_{\mathrm{R}}^2)\widehat{v}_0+(\mu_{\mathrm{R}}^2-\mu_{\mathrm{I}}^2-\lambda_1^2)\widehat{v}_1-(\mu_{\mathrm{R}}-\lambda_1)\widehat{v}_2}{\mu_{\mathrm{I}}(2\mu_{\mathrm{R}}\lambda_1-\mu_{\mathrm{I}}^2-\mu_{\mathrm{R}}^2-\lambda_1^2)}\sin(\mu_{\mathrm{I}}t)\,\mathrm{e}^{\mu_{\mathrm{R}}t}\notag\\
	&=-\frac{|\xi|^{\frac{4}{3}\sigma}\widehat{v}_0+(\eta-1)|\xi|^{\frac{2}{3}\sigma}\widehat{v}_1+\widehat{v}_2}{(\eta-3)|\xi|^{\frac{4}{3}\sigma}}\mathrm{e}^{-|\xi|^{\frac{2}{3}\sigma}t}\notag\\
	&\quad+\frac{(\eta-2)|\xi|^{\frac{4}{3}\sigma}\widehat{v}_0+(\eta-1)|\xi|^{\frac{2}{3}\sigma}\widehat{v}_1+\widehat{v}_2}{(\eta-3)|\xi|^{\frac{4}{3}\sigma}}\cos(C_{\eta}|\xi|^{\frac{2}{3}\sigma}t)\,\mathrm{e}^{\frac{1}{2}(1-\eta)|\xi|^{\frac{2}{3}\sigma }t}\notag\\
	&\quad+\frac{\eta|\xi|^{\frac{4}{3}\sigma}\widehat{v}_0+(\eta+1)|\xi|^{\frac{2}{3}\sigma}\widehat{v}_1+\widehat{v}_2}{2C_{\eta}|\xi|^{\frac{4}{3}\sigma}}\sin(C_{\eta}|\xi|^{\frac{2}{3}\sigma}t)\,\mathrm{e}^{\frac{1}{2}(1-\eta)|\xi|^{\frac{2}{3}\sigma }t}\label{B1}
\end{align}  
with the constant $C_{\eta}:=\frac{1}{2}\sqrt{3+2\eta-\eta^2}>0$ as $\eta\in(0,3)$. \\
\underline{\bf Representation  of solution when $\eta\in(3,\infty)$:} Due to the pairwise distinct characteristic roots, we may employ the general representation (it also holds for any $\eta\neq 3$) as follows:
\begin{align*}
\widehat{v}=\widehat{K}_0\widehat{v}_0+\widehat{K}_1\widehat{v}_1+\widehat{K}_2\widehat{v}_2
\end{align*}
with the kernels in the Fourier space carrying $k=1,2,3$ such that
\begin{align*}
	\widehat{K}_0:=\sum\limits_{j=1,2,3}\frac{\mathrm{e}^{\lambda_jt}\prod\limits_{k\neq j}\lambda_k}{\prod\limits_{k\neq j}(\lambda_j-\lambda_k)},\ \ 
	\widehat{K}_1:=-\sum\limits_{j=1,2,3}\frac{\mathrm{e}^{\lambda_jt}\sum\limits_{k\neq j}\lambda_k}{\prod\limits_{k\neq j}(\lambda_j-\lambda_k)},\ \ 
	\widehat{K}_2:=\sum\limits_{j=1,2,3}\frac{\mathrm{e}^{\lambda_jt}}{\prod\limits_{k\neq j}(\lambda_j-\lambda_k)}.
\end{align*}
 Remark that \eqref{B1} was deduced by the last kernels associated with the conjugate roots $\lambda_{2,3}$. Combining with the characteristic roots and denoting $D_{\eta}:=\frac{1}{2}\sqrt{\eta^2-2\eta-3}>0$ as $\eta\in(3,\infty)$, which fulfills the relation $D_{\eta}+\frac{1}{2}(1-\eta)<0$, the explicit representation of solution in the Fourier space for $\eta\in(3,\infty)$ is given by
\begin{align}
\widehat{v}&=\left[\frac{\mathrm{e}^{-|\xi|^{\frac{2}{3}\sigma}t}}{3-\eta}+\frac{\mathrm{e}^{\frac{1}{2}(1-\eta)|\xi|^{\frac{2}{3}\sigma }t}}{2D_{\eta}}\left(\frac{2D_{\eta}-1+\eta}{3-\eta+2D_{\eta}}\mathrm{e}^{D_{\eta}|\xi|^{\frac{2}{3}\sigma }t}+\frac{2D_{\eta}+1-\eta}{3-\eta-2D_{\eta}}\mathrm{e}^{-D_{\eta}|\xi|^{\frac{2}{3}\sigma }t}\right)\right]\widehat{v}_0\notag\\
&\quad+\left[\frac{(\eta-1)\mathrm{e}^{-|\xi|^{\frac{2}{3}\sigma }t}}{(3-\eta)|\xi|^{\frac{2}{3}\sigma }}+\frac{\mathrm{e}^{\frac{1}{2}(1-\eta)|\xi|^{\frac{2}{3}\sigma }t}}{2D_{\eta}|\xi|^{\frac{2}{3}\sigma }}\left(\frac{2D_{\eta}+1+\eta}{3-\eta+2D_{\eta}}\mathrm{e}^{D_{\eta}|\xi|^{\frac{2}{3}\sigma }t}+\frac{2D_{\eta}-1-\eta}{3-\eta-2D_{\eta}}\mathrm{e}^{-D_{\eta}|\xi|^{\frac{2}{3}\sigma }t}\right)\right]\widehat{v}_1\notag\\
&\quad+\left[\frac{\mathrm{e}^{-|\xi|^{\frac{2}{3}\sigma }t}}{(3-\eta)|\xi|^{\frac{4}{3}\sigma }}+\frac{\mathrm{e}^{\frac{1}{2}(1-\eta)|\xi|^{\frac{2}{3}\sigma }t}}{D_{\eta}|\xi|^{\frac{4}{3}\sigma }}\left(\frac{1}{3-\eta+2D_{\eta}}\mathrm{e}^{D_{\eta}|\xi|^{\frac{2}{3}\sigma }t}-\frac{1}{3-\eta-2D_{\eta}}\mathrm{e}^{-D_{\eta}|\xi|^{\frac{2}{3}\sigma }t}\right)\right]\widehat{v}_2.\label{B2}
\end{align}
\underline{\bf Representation  of solution when $\eta=3$:}  Finally, the identical characteristic roots $\lambda_1=\lambda_2=\lambda_3=-|\xi|^{\frac{2}{3}\sigma}$ when $\eta=3$ imply
\begin{align*}
	\widehat{v}=\left[\widehat{v}_0+\left(|\xi|^{\frac{2}{3}\sigma}\widehat{v}_0+\widehat{v}_1\right)t+\frac{1}{2}\left(|\xi|^{\frac{4}{3}\sigma}\widehat{v}_0+2|\xi|^{\frac{2}{3}\sigma}\widehat{v}_1+\widehat{v}_2\right)t^2\right]\mathrm{e}^{-|\xi|^{\frac{2}{3}\sigma}t}.
\end{align*}

It is well-known that the well-posedness of Cauchy problem is determined by the large frequencies part, consequently we may apply the previous representations of solution to get the following pointwise estimates in the Fourier space:
\begin{align}\label{Critical-Gevrey-data-01}
\chi_{\extt}(\xi)|\widehat{v}|&\lesssim \begin{cases}
\chi_{\extt}(\xi)\,\mathrm{e}^{c|\xi|^{\frac{2}{3}\sigma}t}\left(|\widehat{v}_0|+\langle\xi\rangle^{-\frac{2}{3}\sigma}|\widehat{v}_1|+\langle\xi\rangle^{-\frac{4}{3}\sigma}|\widehat{v}_2|\right)&\mbox{if}\ \ \eta\in(0,1),\\
\chi_{\extt}(\xi)\left(|\widehat{v}_0|+\langle\xi\rangle^{-\frac{2}{3}\sigma}|\widehat{v}_1|+\langle\xi\rangle^{-\frac{4}{3}\sigma}|\widehat{v}_2|\right)&\mbox{if}\ \ \eta=1,\\
\chi_{\extt}(\xi)\,\mathrm{e}^{-c|\xi|^{\frac{2}{3}\sigma}t}\left(|\widehat{v}_0|+\langle\xi\rangle^{-\frac{2}{3}\sigma}|\widehat{v}_1|+\langle\xi\rangle^{-\frac{4}{3}\sigma}|\widehat{v}_2|\right)&\mbox{if}\ \ \eta\in(1,3)\cup(3,\infty),\\
\chi_{\extt}(\xi)\,\mathrm{e}^{-c|\xi|^{\frac{2}{3}\sigma}t}\big(|\widehat{v}_0|+t|\widehat{v}_1|+t^2|\widehat{v}_2|\big)&\mbox{if}\ \ \eta=3,
\end{cases}
\end{align}
with some positive constants $c>0$, where we used $y\,\mathrm{e}^{-2y}\lesssim \mathrm{e}^{-y}$. Then, let us discuss the well-posed property in several situations.
\begin{description}
	\item[Case 1: $\eta\in(0,1)$.] If we assume  any Sobolev regularities for initial data, due to the exponential growth for large frequencies $\chi_{\extt}(\xi)\exp(c|\xi|^{\frac{2}{3}\sigma}t)$ with $c>0$, the Sobolev solution is instable. It leads to the ill-posedness of the linearized Cauchy problem \eqref{Eq-Linear-Third-PDE}, which exactly coincides with \cite[Theorem 3.1]{Bezerra-Carvalho-Santos=2022}. Nevertheless, by assuming the Gevrey-Sobolev regularities for initial data $v_j\in G^{\frac{3}{2\sigma},s+\frac{2(2-j)}{3}\sigma}$ with $j=0,1,2$, namely, $\exp(c\langle\xi\rangle^{\frac{2}{3}\sigma})\langle\xi\rangle^{s+\frac{2(2-j)}{3}\sigma}\widehat{v}_j\in L^2$, one may arrive at
	\begin{align}\label{Eq-01}
		\chi_{\extt}(\xi)\langle\xi\rangle^{s+\frac{2(2-j)}{3}\sigma}\partial_t^j\widehat{v}\in L^2\ \ \mbox{with}\ \ j=0,1,2.
	\end{align} 
Hereinafter, the $L^2$ space concerns the variables $\xi$. Thus, the well-posedness result with the Gevrey-Sobolev data can be proved straightforwardly when $\eta\in(0,1)$.
\item[Case 2: $\eta=1$.] By taking $v_j\in H^{s+\frac{2(2-j)}{3}\sigma}$, i.e. $\langle\xi\rangle^{s+\frac{2(2-j)}{3}\sigma}\widehat{v}_j\in L^2$, for $j=0,1,2$, we notice that \eqref{Eq-01} still holds, which implies the well-posedness of the Cauchy problem \eqref{Eq-Linear-Third-PDE} when $\eta=1$ with the Sobolev data.
\item[Case 3: $\eta\in(1,\infty)$.] Similarly to Case 2, the well-posedness statement for the Sobolev solution is easily achieved by assuming $v_j\in H^{s+\frac{2(2-j)}{3}\sigma}$ with $j=0,1,2$. Furthermore, benefited from the exponential propagator $\chi_{\extt}(\xi)\exp(-c|\xi|^{\frac{2}{3}\sigma}t)$, we may derive
	\begin{align*}
	\chi_{\extt}(\xi)\,\mathrm{e}^{c|\xi|^{\frac{2}{3}\sigma}t}\langle\xi\rangle^{s+\frac{2(2-j)}{3}\sigma}\partial_t^j\widehat{v}\in L^2\ \ \mbox{with}\ \ j=0,1,2.
\end{align*} 
It leads to smoothing effect from the Sobolev data to the general Gevrey solution and the well-posedness of the Cauchy problem \eqref{Eq-Linear-Third-PDE} when $\eta\in(1,\infty)$.
\end{description}
Summarizing the last statements, our proof is totally completed.
\end{proof}
\begin{remark}\label{Rem-eta=1}
	By taking the suitably regular Sobolev data, $\eta=1$ is the important threshold for the Sobolev stability to the linearized Cauchy problem \eqref{Eq-Linear-Third-PDE}, in other words, the model is instable when $\eta\in(0,1)$ but stable when $\eta\in[1,\infty)$. We refer to Figure \ref{imgg}. This phenomenon has been found firstly by \cite{Bezerra-Carvalho-Santos=2022} in the semigroup setting.
\end{remark}
\begin{remark}
	Let us consider the critical index for the Gevrey data, which is an interesting topic in recent years, e.g. the $G^2$ initial data for the 2D Prandtl equation \cite{Gerard-Dormy=2010,Dietert-Gerard=2019}. According to the first statement of Proposition \ref{Prop-Stable}, we also may show the stability of Sobolev solution even when $\eta\in(0,1)$ with the suitable Gevrey-Sobolev data $v_j\in G^{\rho,s+\frac{2(2-j)}{3}\sigma}$ with $j=0,1,2$ carrying $0<\rho\leqslant \frac{3}{2\sigma}$ due to \eqref{Critical-Gevrey-data-01} and the fact that
	\begin{align*}
		\chi_{\extt}(\xi)\exp\left(c|\xi|^{\frac{2}{3}\sigma}t-c|\xi|^{\frac{1}{\rho}}t\right)\lesssim 1\ \ \mbox{only when}\ \ \rho\leqslant\frac{3}{2\sigma}.
	\end{align*}
	That is to say, the Sobolev solution to the linearized Cauchy problem \eqref{Eq-Linear-Third-PDE} with $\eta\in(0,1)$ and Gevrey-Sobolev data becomes instable again when the Gevrey index $\rho>\frac{3}{2\sigma}$, which means that $\rho=\frac{3}{2\sigma}$ is the critical index of the Gevrey-Sobolev data.
\end{remark}
\begin{remark}
	The Gevrey smoothing effect occurs when $\eta\in(1,\infty)$ from the Sobolev data to the Gevrey-Sobolev solution to the linearized Cauchy problem \eqref{Eq-Linear-Third-PDE}, in which $\sigma=\frac{3}{2}$ is the threshold to distinguish different degree of smoothing effect, namely, there exists a Gevrey solution when $\sigma\in(0,\frac{3}{2})$, an analytic solution when $\sigma=\frac{3}{2}$, and an ultra-analytic solution when $\sigma\in(\frac{3}{2},\infty)$.
\end{remark}
\subsection{Some sharp estimates and asymptotic profiles of solutions as $\eta\in(1,\infty)$}
$\ \ \ \ $In this subsection, we derive some sharp $L^2$ estimates of solutions to the linearized Cauchy problem \eqref{Eq-Linear-Third-PDE} with $L^2$ data or $L^2\cap L^1$ data, which will contribute to the proof of Theorem \ref{Thm-GESDS}. Before doing these, let us state some preliminaries.
\begin{lemma}\label{Lem-01}
	Let $n>-2s$ and $\sigma\in(0,\infty)$. The following estimates hold:
	\begin{align}\label{Su02}
		\left\|\chi_{\intt}(\xi)|\xi|^s\mathrm{e}^{-c|\xi|^{\frac{2}{3}\sigma}t}\right\|_{L^2}\lesssim(1+t)^{-\frac{3(2s+n)}{4\sigma}}.
	\end{align}
\end{lemma}
\begin{proof}
Applying polar coordinates and the change of variable $\omega=r^{\frac{2}{3}\sigma}t$, we directly obtain
\begin{align*}
\left\|\chi_{\intt}(\xi)|\xi|^s\mathrm{e}^{-c|\xi|^{\frac{2}{3}\sigma}t}\right\|_{L^2}^2&\lesssim\int_0^{\varepsilon_0}r^{2s+n-1}\mathrm{e}^{-2cr^{\frac{2}{3}\sigma}t}\mathrm{d}r\\
&\lesssim t^{-\frac{3(2s+n)}{2\sigma}}\int_0^{\infty}\omega^{\frac{3(2s+n-1)}{2\sigma}}\mathrm{e}^{-2c\omega}\mathrm{d}\omega^{\frac{3}{2\sigma}}\lesssim t^{-\frac{3(2s+n)}{2\sigma}}
\end{align*}
when $t\in[1,\infty)$ due to $2s+n-1>-1$, and bounded estimates hold for $t\in[0,1]$. Then, our proof is completed.
\end{proof}
\begin{lemma}\label{Lem-02}
	Let $n>\frac{4}{3}\sigma$ and $\sigma\in(0,\infty)$. The following estimates hold:
	\begin{align}
	\left\|\chi_{\intt}(\xi)|\xi|^{-\frac{4}{3}\sigma}\left(\mathrm{e}^{-|\xi|^{\frac{2}{3}\sigma}t}-\cos(C_{\eta}|\xi|^{\frac{2}{3}\sigma}t)\,\mathrm{e}^{-\frac{1}{2}(\eta-1)|\xi|^{\frac{2}{3}\sigma}t}\right)\right\|_{L^2}&\lesssim (1+t)^{-\frac{3n-8\sigma}{4\sigma}},\label{A1}\\
	\left\|\chi_{\intt}(\xi)|\xi|^{-\frac{4}{3}\sigma}\sin(C_{\eta}|\xi|^{\frac{2}{3}\sigma}t)\,\mathrm{e}^{-\frac{1}{2}(\eta-1)|\xi|^{\frac{2}{3}\sigma}t}\right\|_{L^2}&\lesssim (1+t)^{-\frac{3n-8\sigma}{4\sigma}},\label{A2}
	\end{align}
where $\eta\in(1,3)$ and $C_{\eta}>0$.
\end{lemma}
\begin{proof}
According to $1-\cos y=2\sin^2\frac{y}{2}$, we notice
\begin{align*}
&|\xi|^{-\frac{4}{3}\sigma}\left(\mathrm{e}^{-|\xi|^{\frac{2}{3}\sigma}t}-\cos(C_{\eta}|\xi|^{\frac{2}{3}\sigma}t)\,\mathrm{e}^{-\frac{1}{2}(\eta-1)|\xi|^{\frac{2}{3}\sigma}t}\right)\\
&\qquad=|\xi|^{-\frac{4}{3}\sigma}\mathrm{e}^{-\frac{1}{2}(\eta-1)|\xi|^{\frac{2}{3}\sigma}t}\left(\mathrm{e}^{\frac{1}{2}(\eta-3)|\xi|^{\frac{2}{3}\sigma}t}-1\right)+2|\xi|^{-\frac{4}{3}\sigma}\sin^2\left(\frac{C_{\eta}}{2}|\xi|^{\frac{2}{3}\sigma }t\right)\mathrm{e}^{-\frac{1}{2}(\eta-1)|\xi|^{\frac{2}{3}\sigma}t}\\
&\qquad=\frac{(\eta-3)t}{2|\xi|^{\frac{2}{3}\sigma}}\mathrm{e}^{-\frac{1}{2}(\eta-1)|\xi|^{\frac{2}{3}\sigma}t}\int_0^1\mathrm{e}^{-\frac{1}{2}(3-\eta)|\xi|^{\frac{2}{3}\sigma} t\tau}\mathrm{d}\tau+2t^2\left|\frac{\sin(\frac{C_{\eta}}{2}|\xi|^{\frac{2}{3}\sigma}t)}{|\xi|^{\frac{2}{3}\sigma}t}\right|^2\mathrm{e}^{-\frac{1}{2}(\eta-1)|\xi|^{\frac{2}{3}\sigma}t}.
\end{align*}
Due to the fact that $|\sin y|\lesssim |y|$, it yields
\begin{align*}
\mbox{LHS of }\eqref{A1}&\lesssim t\left\|\chi_{\intt}(\xi)|\xi|^{-\frac{2}{3}\sigma}\mathrm{e}^{-c|\xi|^{\frac{2}{3}\sigma}t}\right\|_{L^2}+t^2\left\|\chi_{\intt}(\xi)\mathrm{e}^{-c|\xi|^{\frac{2}{3}\sigma}t}\right\|_{L^2}\lesssim (1+t)^{-\frac{3n-8\sigma}{4\sigma}}
\end{align*}
with $n>\frac{4}{3}\sigma$, where we applied Lemma \ref{Lem-01} with $s=-\frac{2}{3}\sigma$ and $s=0$. Similarly,
\begin{align*}
\mbox{LHS of }\eqref{A2}&\lesssim t\left\|\chi_{\intt}(\xi)|\xi|^{-\frac{2}{3}\sigma}\frac{\sin(C_{\eta}|\xi|^{\frac{2}{3}\sigma}t)}{|\xi|^{\frac{2}{3}\sigma}t}\mathrm{e}^{-c|\xi|^{\frac{2}{3}\sigma}t}\right\|_{L^2}\lesssim(1+t)^{-\frac{3n-8\sigma}{4\sigma}},
\end{align*}
again with $n>\frac{4}{3}\sigma$. The proof is finished.
\end{proof}

\begin{lemma}\label{Lem-03}
Under the corresponding conditions to those of Lemma \ref{Lem-01} and Lemma \ref{Lem-02}, respectively, all estimates from 
the above mentioned are sharp for large-time $t\gg1$.
\end{lemma}
\begin{proof}
The upper bounds estimates have been derived in Lemma \ref{Lem-01} as well as Lemma \ref{Lem-02}. For this reason, we just need to get the lower bounds estimates with the same time-dependent coefficients as the upper bounds when $t\gg1$. Again with the variable $\omega=r^{\frac{2}{3}\sigma}t$, we estimate
\begin{align*}
[\mbox{LHS of }\eqref{Su02}]^2\gtrsim t^{-\frac{3(2s+n)}{2\sigma}}\int_0^{\varepsilon_0^{\frac{2}{3}\sigma}t}\omega^{\frac{3(2s+n)}{2\sigma}-1}\,\mathrm{e}^{-2c\omega}\mathrm{d}\omega\gtrsim t^{-\frac{3(2s+n)}{2\sigma}}
\end{align*}
for large-time $t\gg1$ such that $\varepsilon_0^{\frac{2}{3}\sigma}t\geqslant 1$, where we shrink the domain $(0,\varepsilon_0^{\frac{2}{3}\sigma}t)$ into $(0,1)$ and used $2s+n>0$ from the assumption of Lemma \ref{Lem-01}. Similarly, we know
\begin{align*}
[\mbox{LHS of } \eqref{A1}]^2&\gtrsim\int_0^{\varepsilon_0}r^{-\frac{8}{3}\sigma+n-1}\left(\mathrm{e}^{-r^{\frac{2}{3}\sigma}t}-\cos(C_{\eta}r^{\frac{2}{3}\sigma}t)\,\mathrm{e}^{-\frac{1}{2}(\eta-1)r^{\frac{2}{3}\sigma}t}\right)^2\mathrm{d}r\\
&\gtrsim t^{-\frac{3n-8\sigma}{2\sigma}}\int_0^{\varepsilon_0^{\frac{2}{3}\sigma}t}\omega^{\frac{3n-4\sigma}{2\sigma}-1}\omega^{-2}\left(\mathrm{e}^{-\omega}-\cos(C_{\eta}\,\omega)\,\mathrm{e}^{-\frac{1}{2}(\eta-1)\omega}\right)^2\mathrm{d}\omega.
\end{align*}
Due to the fact that
\begin{align*}
\lim\limits_{\omega\downarrow 0}\omega^{-1}\left(\mathrm{e}^{-\omega}-\cos(C_{\eta}\,\omega)\,\mathrm{e}^{-\frac{1}{2}(\eta-1)\omega}\right)=\frac{1}{2}(\eta-3)<0,
\end{align*}
there exists $\varepsilon_0^*>0$ with $0<\omega<\varepsilon_0^*<\varepsilon_0^{\frac{2}{3}\sigma}t$ for $t\gg1$ such that
\begin{align*}
[\mbox{LHS of } \eqref{A1}]^2&\gtrsim t^{-\frac{3n-8\sigma}{2\sigma}}\int_0^{\varepsilon_0^*}\omega^{\frac{3n-4\sigma}{2\sigma}-1}\mathrm{d}\omega \gtrsim t^{-\frac{3n-8\sigma}{2\sigma}},
\end{align*}
where we employed our assumption $3n-4\sigma>0$ in Lemma \ref{Lem-02}. Finally, the same idea leads to
\begin{align*}
[\mbox{LHS of } \eqref{A2}]^2&\gtrsim t^{-\frac{3n-8\sigma}{2\sigma}}\int_0^{\varepsilon_0^{\frac{2}{3}\sigma}t}\omega^{\frac{3n-4\sigma}{2\sigma}-1}\left|\frac{\sin(C_{\eta}\,\omega)}{C_{\eta}\,\omega}\right|^2\mathrm{e}^{-2c\omega}\mathrm{d}\omega \gtrsim t^{-\frac{3n-8\sigma}{2\sigma}}
\end{align*}
as $t\gg1$ and $3n-4\sigma>0$. Therefore, our proof is finished.
\end{proof}

Because of the representations of solutions for different size of $\eta\in(1,\infty)$, we next will state some sharp estimates and asymptotic profiles of solutions in the cases $\eta\in(1,3)$,   $\eta\in(3,\infty)$ and $\eta=3$, respectively. We postpone the proof of sharpness for these estimates in Proposition \ref{Prop-sharpness-linear}. Let us recall the data space $\ml{B}_{\sigma}$ introduced in \eqref{Data-space-D}. We first study the case $\eta\in(1,3)$.
\begin{prop}\label{Prop-(1,3)}
	Let $n\geqslant1$ and $n>\frac{4}{3}\sigma$ with $\sigma\in(0,\infty)$. Let us consider the linearized Cauchy problem \eqref{Eq-Linear-Third-PDE} with $\eta\in(1,3)$ and $(v_0,v_1,v_2)\in\ml{B}_{\sigma}$. The solution fulfills the following estimates:
	\begin{align}
	\|v(t,\cdot)\|_{L^2}&\lesssim (1+t)^{-\frac{3n-8\sigma}{4\sigma}}\|(v_0,v_1,v_2)\|_{\ml{B}_{\sigma}},\label{C(1,3),L2L1-L2}\\
	\|v(t,\cdot)\|_{\dot{H}^{\frac{4}{3}\sigma}}&\lesssim (1+t)^{-\frac{3n}{4\sigma}}\|(v_0,v_1,v_2)\|_{\ml{B}_{\sigma}}.\label{C(1,3),HsL1-L2}
	\end{align}
	Furthermore, the asymptotic profile  when $\eta\in(1,3)$ is described by $w_{(1,3)}=w_{(1,3)}(t,x)$ such that
	\begin{align*}
	w_{(1,3)}(t,x)&:=\ml{F}_{\xi\to x}^{-1}\left[|\xi|^{-\frac{4}{3}\sigma}\left(-\mathrm{e}^{-|\xi|^{\frac{2}{3}\sigma}t}+\cos(C_{\eta}|\xi|^{\frac{2}{3}\sigma}t)\,\mathrm{e}^{-\frac{1}{2}(\eta-1)|\xi|^{\frac{2}{3}\sigma}t}\right)\right]\frac{v_2(x)}{\eta-3}\\
	&\ \quad+\ml{F}^{-1}_{\xi\to x}\left[|\xi|^{-\frac{4}{3}\sigma}\sin(C_{\eta}|\xi|^{\frac{2}{3}\sigma}t)\,\mathrm{e}^{-\frac{1}{2}(\eta-1)|\xi|^{\frac{2}{3}\sigma}t}\right]\frac{v_2(x)}{2C_{\eta}}
	\end{align*}
	in the sense of refined estimates
	\begin{align}\label{Refine(1,3)-1}
	\|v(t,\cdot)-w_{(1,3)}(t,\cdot)\|_{L^2}&\lesssim (1+t)^{-\frac{3n-4\sigma}{4\sigma}}\|(v_0,v_1,v_2)\|_{\ml{B}_{\sigma}},\\
		\|v(t,\cdot)-w_{(1,3)}(t,\cdot)\|_{\dot{H}^{\frac{4}{3}\sigma}}&\lesssim (1+t)^{-\frac{3n+4\sigma}{4\sigma}}\|(v_0,v_1,v_2)\|_{\ml{B}_{\sigma}}.\label{Refine(1,3)-2}
	\end{align}
	Assuming $v_0\equiv0\equiv v_1$, the solution fulfills the following bounded estimate:
	\begin{align}\label{C(1,3),L2-L2}
	\|v(t,\cdot)\|_{\dot{H}^{\frac{4}{3}\sigma}}\lesssim\|v_2\|_{L^2}.
	\end{align}
\end{prop}
\begin{proof}
Recalling the representation \eqref{B1} as $\eta\in(1,3)$, we are able to derive
\begin{align*}
\|\chi_{\intt}(\xi)\widehat{v}(t,\xi)\|_{L^2}&\lesssim\left\|\chi_{\intt}(\xi)\,\mathrm{e}^{-c|\xi|^{\frac{2}{3}\sigma}t}\right\|_{L^2}\|\widehat{v}_0\|_{L^{\infty}}+\left\|\chi_{\intt}(\xi)|\xi|^{-\frac{2}{3}\sigma}\mathrm{e}^{-c|\xi|^{\frac{2}{3}\sigma}t}\right\|_{L^2}\|\widehat{v}_1\|_{L^{\infty}}\\
&\quad+\left\|\chi_{\intt}(\xi)|\xi|^{-\frac{4}{3}\sigma}\left(\mathrm{e}^{-|\xi|^{\frac{2}{3}\sigma}t}-\cos(C_{\eta}|\xi|^{\frac{2}{3}\sigma}t)\,\mathrm{e}^{-\frac{1}{2}(\eta-1)|\xi|^{\frac{2}{3}\sigma}t}\right)\right\|_{L^2}\|\widehat{v}_2\|_{L^{\infty}}\\
&\quad+\left\|\chi_{\intt}(\xi)|\xi|^{-\frac{4}{3}\sigma}\sin(C_{\eta}|\xi|^{\frac{2}{3}\sigma}t)\,\mathrm{e}^{-\frac{1}{2}(\eta-1)|\xi|^{\frac{2}{3}\sigma}t}\right\|_{L^2}\|\widehat{v}_2\|_{L^{\infty}}.
\end{align*}
With the aid of $s=0,-\frac{2}{3}\sigma$ in Lemma \ref{Lem-01}, \eqref{A1}-\eqref{A2} in Lemma \ref{Lem-02}, as well as the Hausdorff-Young inequality $\|\widehat{v}_j\|_{L^{\infty}}\lesssim\|v_j\|_{L^1}$ with $j=0,1,2$, one may estimate
\begin{align*}
\|\chi_{\intt}(\xi)\widehat{v}(t,\xi)\|_{L^2}\lesssim (1+t)^{-\frac{3n-8\sigma}{4\sigma}}\|(v_0,v_1,v_2)\|_{L^1\times L^1\times L^1},
\end{align*}
with $n>\frac{4}{3}\sigma$. For another, it is clear that
\begin{align*}
\left\|\chi_{\extt}(\xi)\widehat{v}(t,\xi)\right\|_{L^2}&\lesssim\left\|\chi_{\extt}(\xi)\,\mathrm{e}^{-c|\xi|^{\frac{2}{3}\sigma}t}\left(\widehat{v}_0(\xi)+\langle\xi\rangle^{-\frac{2}{3}\sigma}\widehat{v}_1(\xi)+\langle\xi\rangle^{-\frac{4}{3}\sigma}\widehat{v}_2(\xi)\right)\right\|_{L^2}\\
&\lesssim\mathrm{e}^{-ct}\|(v_0,v_1,v_2)\|_{L^2\times L^2\times L^2}.
\end{align*}
An application of the Parseval equality associated with the last two estimates completes the proof of \eqref{C(1,3),L2L1-L2}. Concerning the desired estimate \eqref{C(1,3),HsL1-L2}, we follow the analogous method as the above one to get
\begin{align*}
\|v(t,\cdot)\|_{\dot{H}^{\frac{4}{3}\sigma}}&\lesssim\left\|\chi_{\intt}(\xi)\,\mathrm{e}^{-c|\xi|^{\frac{2}{3}\sigma}t}\right\|_{L^2}\|(v_0,v_1,v_2)\|_{L^1\times L^1\times L^1}\\
&\quad+\left\|\chi_{\extt}(\xi)\,\mathrm{e}^{-c|\xi|^{\frac{2}{3}\sigma}t}\right\|_{L^{\infty}}\|(v_0,v_1,v_2)\|_{H^{\frac{4}{3}\sigma}\times H^{\frac{2}{3}\sigma}\times L^2}\\
&\lesssim (1+t)^{-\frac{3n}{4\sigma}}\|(v_0,v_1,v_2)\|_{L^1\times L^1\times L^1}+\mathrm{e}^{-ct}\|(v_0,v_1,v_2)\|_{H^{\frac{4}{3}\sigma}\times H^{\frac{2}{3}\sigma}\times L^2},
\end{align*}
where we applied Lemma \ref{Lem-01} with $s=0$. Furthermore, taking $\hat{v}_0\equiv0\equiv \hat{v}_1$, we can obtain
\begin{align*}
|\xi|^{\frac{4}{3}\sigma}|\widehat{v}|\lesssim \left(1+|\cos(C_{\eta}|\xi|^{\frac{2}{3}\sigma }t)|+|\sin(C_{\eta}|\xi|^{\frac{2}{3}\sigma }t)|\right)\mathrm{e}^{-c|\xi|^{\frac{2}{3}\sigma}t}|\widehat{v}_2|\lesssim |\widehat{v}_2|,
\end{align*}
which implies \eqref{C(1,3),L2-L2} immediately. Finally, benefited from the representation \eqref{B1}, we know
\begin{align*}
|\widehat{v}-\widehat{w}_{(1,3)}|\lesssim\mathrm{e}^{-c|\xi|^{\frac{2}{3}\sigma}t}\left(|\widehat{v}_0|+|\xi|^{-\frac{2}{3}\sigma}|\widehat{v}_1|\right).
\end{align*}
By applying Lemma \ref{Lem-01} with $s=-\frac{2}{3}\sigma$, the desired refined estimate can be proved directly.
\end{proof}
\begin{remark}
	Comparing \eqref{C(1,3),L2L1-L2}, \eqref{C(1,3),HsL1-L2} with \eqref{Refine(1,3)-1} and \eqref{Refine(1,3)-2}, respectively, we notice that the decay rates have been improved $(1+t)^{-1}$ by subtracting the function $w_{(1,3)}(t,\cdot)$ in the $L^2$ and $\dot{H}^{\frac{4}{3}\sigma}$ norms. In other words, the general diffusion-waves function $w_{(1,3)}$ is the asymptotic profile of $v$ to the linearized Cauchy problem \eqref{Eq-Linear-Third-PDE} with $\eta\in(1,3)$.
\end{remark}

Let us turn to the second situation $\eta\in(3,\infty)$, which will be separated into two results with respect to the size of dimensions.

\begin{prop}\label{Prop-(3,infty)}
	Let $n\geqslant 1$ and $n>\frac{8}{3}\sigma$ with $\sigma\in(0,\infty)$. Let us consider the linearized Cauchy problem \eqref{Eq-Linear-Third-PDE} with $\eta\in(3,\infty)$ and  $(v_0,v_1,v_2)\in\ml{B}_{\sigma}$. The solution fulfills the estimates \eqref{C(1,3),L2L1-L2}-\eqref{C(1,3),HsL1-L2}. Furthermore, the asymptotic profile  when $\eta\in(3,\infty)$ is described by $w_{(3,\infty)}=w_{(3,\infty)}(t,x)$ such that
	\begin{align*}
	w_{(3,\infty)}(t,x)&:=\ml{F}^{-1}_{\xi\to x}\left[\frac{\mathrm{e}^{-|\xi|^{\frac{2}{3}\sigma }t}}{(3-\eta)|\xi|^{\frac{4}{3}\sigma }}+\frac{\mathrm{e}^{\frac{1}{2}(1-\eta)|\xi|^{\frac{2}{3}\sigma }t}}{D_{\eta}|\xi|^{\frac{4}{3}\sigma }}\left(\frac{\mathrm{e}^{D_{\eta}|\xi|^{\frac{2}{3}\sigma }t}}{3-\eta+2D_{\eta}}-\frac{\mathrm{e}^{-D_{\eta}|\xi|^{\frac{2}{3}\sigma }t}}{3-\eta-2D_{\eta}}\right)\right]v_2(x)
	\end{align*}
	in the sense of refined estimates
	\begin{align*}
	\|v(t,\cdot)-w_{(3,\infty)}(t,\cdot)\|_{L^2}&\lesssim (1+t)^{-\frac{3n-4\sigma}{4\sigma}}\|(v_0,v_1,v_2)\|_{\ml{B}_{\sigma}},\\
	\|v(t,\cdot)-w_{(3,\infty)}(t,\cdot)\|_{\dot{H}^{\frac{4}{3}\sigma}}&\lesssim (1+t)^{-\frac{3n+4\sigma}{4\sigma}}\|(v_0,v_1,v_2)\|_{\ml{B}_{\sigma}}.
	\end{align*}
	Assuming $v_0\equiv0\equiv v_1$, the solution fulfills the estimate \eqref{C(1,3),L2-L2}.
\end{prop}
\begin{proof}
The representation \eqref{B2} shows the pointwise estimates
\begin{align}
|\widehat{v}|&\lesssim\mathrm{e}^{-c|\xi|^{\frac{2}{3}\sigma}t}\left(|\widehat{v}_0|+|\xi|^{-\frac{2}{3}\sigma}|\widehat{v}_1|+|\xi|^{-\frac{4}{3}\sigma}|\widehat{v}_2|\right),\label{NEW001}\\
|\widehat{v}-\widehat{w}_{(3,\infty)}|&\lesssim\mathrm{e}^{-c|\xi|^{\frac{2}{3}\sigma}t}\left(|\widehat{v}_0|+|\xi|^{-\frac{2}{3}\sigma}|\widehat{v}_1|\right).\notag
\end{align}
Following similar approaches to those of Proposition \ref{Prop-(1,3)}, we complete the proof.
\end{proof}
\begin{remark}
	We observe that the decay rates have been improved $(1+t)^{-1}$ by subtracting the function $w_{(3,\infty)}(t,\cdot)$ in the $L^2$ and $\dot{H}^{\frac{4}{3}\sigma}$ norms. Due to the Fourier multipliers $\exp(\alpha_0|\xi|^{\frac{2}{3}\sigma}t)$ with $\alpha_0=-1$ and $\alpha_0=\frac{1}{2}(1-\eta)\pm D_{\eta}<0$, the general diffusion function $w_{(3,\infty)}$ is the asymptotic profile of $v$ to the linearized Cauchy problem \eqref{Eq-Linear-Third-PDE} with $\eta\in(3,\infty)$.
\end{remark}

\begin{coro}\label{Coro-New}
Let $n\geqslant 1$ and $\frac{4}{3}\sigma<n\leqslant\frac{8}{3}\sigma$ with $\sigma\in(0,\infty)$. Let us consider the linearized Cauchy problem \eqref{Eq-Linear-Third-PDE} with $\eta\in(3,\infty)$ and  $(v_0,v_1,v_2)\in\ml{B}_{\sigma}$. The solution fulfills the following estimates:
\begin{align*}
\|v(t,\cdot)\|_{L^2}\lesssim\begin{cases}
(1+t)^{-\frac{3n-8\sigma}{4\sigma}}\|(v_0,v_1,v_2)\|_{\ml{B}_{\sigma}}&\mbox{if}\ \ \frac{4}{3}\sigma<n<\frac{8}{3}\sigma,\\
\ln(\mathrm{e}+t)\|(v_0,v_1,v_2)\|_{\ml{B}_{\sigma}}&\mbox{if}\ \ n=\frac{8}{3}\sigma.
\end{cases}
\end{align*}
\end{coro}
\begin{proof}
By taking time-derivative to \eqref{B2}, we obtain
\begin{align*}
|\widehat{v}_t|\lesssim\mathrm{e}^{-c|\xi|^{\frac{2}{3}\sigma}t}\left(|\xi|^{\frac{2}{3}\sigma}|\widehat{v}_0|+|\widehat{v}_1|+|\xi|^{-\frac{2}{3}\sigma}|\widehat{v}_2|\right),
\end{align*}
which immediately shows
\begin{align*}
\|v_t(t,\cdot)\|_{L^2}\lesssim(1+t)^{-\frac{3n-4\sigma}{4\sigma}}\|(v_0,v_1,v_2)\|_{\ml{B}_{\sigma}}
\end{align*}
by assuming $n>\frac{4}{3}\sigma$ since Lemma \ref{Lem-01} with $s=-\frac{2}{3}\sigma$. Due to the  relation $v(t,x)=\int_0^tv_t(\tau,x)\mathrm{d}\tau+v_0(x)$, one gets
\begin{align*}
\|v(t,\cdot)\|_{L^2}&\lesssim\int_0^t\|v_t(\tau,\cdot)\|_{L^2}\mathrm{d}\tau+\|v_0\|_{L^2}\lesssim\left(1+\int_0^t(1+\tau)^{-\frac{3n-4\sigma}{4\sigma}}\mathrm{d}\tau\right)\|(v_0,v_1,v_2)\|_{\ml{B}_{\sigma}}.
\end{align*}
From $\frac{4}{3}\sigma<n\leqslant\frac{8}{3}\sigma$, i.e. $0<\frac{3n-4\sigma}{4\sigma}\leqslant 1$, we obtain the desired estimates.
\end{proof}

\begin{remark}
	In comparison with the diffusion-waves case $\eta\in(1,3)$, due to the lack of oscillations $\sin(C_{\eta}|\xi|^{\frac{2}{3}\sigma}t)$ and $\cos(C_{\eta}|\xi|^{\frac{2}{3}\sigma}t)$ in the solution formula from the diffusion case $\eta\in(3,\infty)$, we have a logarithmic loss in the limit case $n=\frac{8}{3}\sigma$, but the estimates in Proposition \ref{Prop-(3,infty)} and Corollary \ref{Coro-New} for the remaining case $n\in(\frac{4}{3}\sigma,\frac{8}{3}\sigma)\cup(\frac{8}{3}\sigma,\infty)$ are exactly the same as those of Proposition \ref{Prop-(1,3)}.
\end{remark}

We focus on the final scenario $\eta=3$.

\begin{prop}\label{Prop-(3)}
	Let $n\geqslant 1$ with $\sigma\in(0,\infty)$. Let us consider the linearized Cauchy problem \eqref{Eq-Linear-Third-PDE} with $\eta=3$ and $(v_0,v_1,v_2)\in\ml{B}_{\sigma}$. The solution fulfills the estimates \eqref{C(1,3),L2L1-L2}-\eqref{C(1,3),HsL1-L2}. Furthermore, the asymptotic profile when $\eta=3$ is described by $w_{[3]}=w_{[3]}(t,x)$ such that
	\begin{align*}
	w_{[3]}(t,x)&:=\frac{t^2}{2}\ml{F}^{-1}_{\xi\to x}\left(\mathrm{e}^{-|\xi|^{\frac{2}{3}\sigma }t}\right)v_2(x)
	\end{align*}
	in the sense of refined estimates
	\begin{align*}
	\|v(t,\cdot)-w_{[3]}(t,\cdot)\|_{L^2}&\lesssim (1+t)^{-\frac{3n-4\sigma}{4\sigma}}\|(v_0,v_1,v_2)\|_{\ml{B}_{\sigma}},\\
	\|v(t,\cdot)-w_{[3]}(t,\cdot)\|_{\dot{H}^{\frac{4}{3}\sigma}}&\lesssim (1+t)^{-\frac{3n+4\sigma}{4\sigma}}\|(v_0,v_1,v_2)\|_{\ml{B}_{\sigma}}.
	\end{align*}
	Assuming $v_0\equiv0\equiv v_1$, the solution fulfills the estimate \eqref{C(1,3),L2-L2}.
\end{prop}

\begin{proof}
From the pointwise estimates
\begin{align*}
|\widehat{v}|&\lesssim\mathrm{e}^{-c|\xi|^{\frac{2}{3}\sigma}t}\left(|\widehat{v}_0|+t|\widehat{v}_1|+t^2|\widehat{v}_2|\right),\\
|\widehat{v}-\widehat{w}_{[3]}|&\lesssim\mathrm{e}^{-c|\xi|^{\frac{2}{3}\sigma}t}\,\big(|\widehat{v}_0|+t|\widehat{v}_1|\big),
\end{align*}
we may complete the proof easily by using Lemma \ref{Lem-01}.
\end{proof}
\begin{remark}
	Because the disappearance of singularities $|\xi|^{-\frac{2}{3}\sigma}$ and $|\xi|^{-\frac{4}{3}\sigma}$ for small frequencies in the representation of solution when $\eta=3$, we can get rid of the restriction of $n$ with respect to $\sigma$, moreover, the asymptotic profile $w_{[3]}$ does not contain any singular component.
\end{remark}
\begin{remark}\label{Rem-eta=3}
Let us combine Propositions \ref{Prop-(1,3)}-\ref{Prop-(3)} and Corollary \ref{Coro-New}. We realize that the solution to the linearized Cauchy problem \eqref{Eq-Linear-Third-PDE} for $n\geqslant 1$ fulfills the same $(L^2\cap L^1)-L^2$ type estimates \eqref{C(1,3),L2L1-L2}-\eqref{C(1,3),HsL1-L2} and $L^2-L^2$ type estimate \eqref{C(1,3),L2-L2} under the restriction  on dimensions 
\begin{align}\label{Restriction(n,sigma)}
\begin{cases}
\displaystyle{n>\frac{4}{3}\sigma}&\mbox{if}\ \ \eta\in(1,3),\\[0.5em]
n>0&\mbox{if}\ \ \eta=3,\\[0.5em]
\displaystyle{n>\frac{4}{3}\sigma \ \mbox{ but }\ n\neq\frac{8}{3}\sigma}&\mbox{if}\ \ \eta\in(3,\infty),
\end{cases}
\end{align}
for any $\sigma\in(0,\infty)$. This phenomenon is caused by
\begin{itemize}
	\item the general diffusion profile $w_{(3,\infty)}(t,x)$ with strong singularity $|\xi|^{-\frac{4}{3}\sigma}$ when $\eta\in(3,\infty)$;
	\item the general diffusion-waves profile $w_{(1,3)}(t,x)$ with strong singularity $|\xi|^{-\frac{4}{3}\sigma}$ and oscillations $\sin(C_{\eta}|\xi|^{\frac{2}{3}\sigma}t)$ as well as $\cos(C_{\eta}|\xi|^{\frac{2}{3}\sigma}t)$ when $\eta\in(0,3)$;
	\item the general diffusion profile $w_{[3]}(t,x)$ without any singularity when $\eta=3$.
\end{itemize}
   Namely, $\eta=3$ is the  threshold to distinguish different kinds of asymptotic profiles (see Figure \ref{imgg} 
    specifically). Note that different degree of singularities with respect to small $|\xi|$ will be generated by various asymptotic profiles.
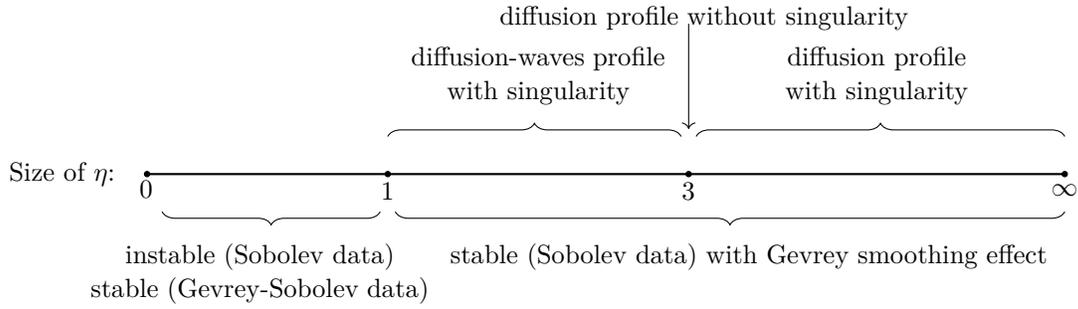
\begin{figure}[http]
	\centering
	\begin{tikzpicture}
	\draw[thick] (-0.2,0) -- (12,0) node[below] {$\infty$};
	\node[left] at (0,-0.2) {{$0$}};
	\node[below] at (3,0) {{$1$}};
	\node[below] at (7,0) {{$3$}};
	\draw [decorate, decoration={brace, amplitude=5pt}]  (3,0.5) -- (6.9,0.5);
	\draw [decorate, decoration={brace, amplitude=5pt}]  (7.1,0.5) -- (12,0.5);
	\draw [decorate, decoration={brace, amplitude=5pt}]  (2.9,-0.5)--(0,-0.5) ;
	\draw [decorate, decoration={brace, amplitude=5pt}]  (12,-0.5)--(3.1,-0.5) ;
		\node[above] at (5,0.8) {with singularity};
	\node[above] at (5,1.25) {diffusion-waves profile};
	\node[above] at (9.5,1.25) {diffusion profile};
			\node[above] at (9.5,0.8) {with singularity};
			\node[above] at (7.2,1.8){diffusion profile without singularity};
	\node[above] at (1.3,-1.4) {instable (Sobolev data)};
		\node[above] at (1.3,-1.85) {stable (Gevrey-Sobolev data)};
	\node[above] at (7.8,-1.4) {stable (Sobolev data) with Gevrey smoothing effect};
	\node[left] at (-0.5,0) {Size of $\eta$:};
	\draw[fill] (-0.2, 0) circle[radius=1pt];
	\draw[fill] (3, 0) circle[radius=1pt];
	\draw[fill] (7, 0) circle[radius=1pt];
	\draw[fill] (12, 0) circle[radius=1pt];
	\draw[->] (7,2) -- (7,0.6);
	\end{tikzpicture}
	\caption{Two thresholds with respect to $\eta$}
	\label{imgg}
\end{figure}
\end{remark}

To end this section, we will state that the derived $(L^2\cap L^1)$ type estimates in Propositions \ref{Prop-(1,3)}-\ref{Prop-(3)} are really sharp in the sense of same large-time behaviors for upper bounds estimates and lower bound estimates. The proof is motived by \cite{Ikehata=2014,Ikehata-Takeda=2019}.
\begin{prop}\label{Prop-sharpness-linear}
Let $n\geqslant 1$ and $\sigma\in(0,\infty)$. Let us consider the linearized Cauchy problem \eqref{Eq-Linear-Third-PDE} with $n>\frac{4}{3}\sigma$ if $\eta\in(1,3)$, $n>0$ if $\eta=3$, $n>\frac{8}{3}\sigma$ if $\eta\in(3,\infty)$, and $(v_0,v_1,v_2)\in\ml{B}_{\sigma}$. The solution fulfills the following sharp estimates:
\begin{align}
\|v(t,\cdot)\|_{L^2}\simeq t^{-\frac{3n-8\sigma}{4\sigma}}\ \ \mbox{and}\ \ \|v(t,\cdot)\|_{\dot{H}^{\frac{4}{3}\sigma}}\simeq t^{-\frac{3n}{4\sigma}}\label{Su05}
\end{align}
for large-time $t\gg1$ if $P_{v_2}:=\int_{\mb{R}^n}v_2(x)\mathrm{d}x\neq0$.
\end{prop}
\begin{proof}
Let us rewrite the asymptotic profiles by
\begin{align*}
w_{\Xi}(t,x)=\ml{L}_{\Xi}(t,x)v_2(x)
\end{align*}
with $\Xi\in\{(1,3),[3],(3,\infty)\}$. In other words, Propositions \ref{Prop-(1,3)}-\ref{Prop-(3)} have proved
\begin{align}
\left\|\big(v-\ml{L}_{\Xi}v_2\big)(t,\cdot)\right\|_{L^2}\lesssim(1+t)^{-\frac{3n-8\sigma}{4\sigma}-1}\|(v_0,v_1,v_2)\|_{\ml{B}_{\sigma}}.\label{Su03}
\end{align}
By applying the Lagrange theorem
\begin{align*}
	|\ml{L}_{\Xi}(t,x-y)-\ml{L}_{\Xi}(t,x)|\lesssim|y|\,|\nabla\ml{L}_{\Xi}(t,x-\theta_1y)|
\end{align*}
with $\theta_1\in(0,1)$, we are able to obtain
\begin{align*}
\|\ml{L}_{\Xi}(t,D)v_2(\cdot)-\ml{L}_{\Xi}(t,\cdot)P_{v_2}\|_{L^2}&\lesssim\left\|\int_{|y|\leqslant t^{\alpha_1}}\big(\ml{L}_{\Xi}(t,\cdot-y)-\ml{L}_{\Xi}(t,\cdot)\big)v_2(y)\mathrm{d}y\right\|_{L^2}\\
&\quad+\left\|\int_{|y|\geqslant t^{\alpha_1}}\big(|\ml{L}_{\Xi}(t,\cdot-y)|+|\ml{L}_{\Xi}(t,\cdot)|\big)|v_2(y)|\mathrm{d}y\right\|_{L^2}\\
&\lesssim t^{\alpha_1}\|\nabla\ml{L}_{\Xi}(t,\cdot)\|_{L^2}\|v_2\|_{L^1}+\|\ml{L}_{\Xi}(t,\cdot)\|_{L^2}\|v_2\|_{L^1(|x|\geqslant t^{\alpha_1})}
\end{align*}
with a positive small constant $0<\alpha_1\ll 1$. Due to $v_2\in L^1$ implying $\|v_2\|_{L^1(|x|\geqslant t^{\alpha_1})}=o(1)$, by using Lemmas \ref{Lem-01} and \ref{Lem-02}, we derive
\begin{align}\label{Su04} 
\|\ml{L}_{\Xi}(t,D)v_2(\cdot)-\ml{L}_{\Xi}(t,\cdot)P_{v_2}\|_{L^2}=o(t^{-\frac{3n-8\sigma}{4\sigma}})
\end{align}
as $t\gg1$. According to \eqref{Su03}-\eqref{Su04} and the Minkowski inequality, one claims
\begin{align*}
	\|v(t,\cdot)\|_{L^2}&\gtrsim\|\ml{L}_{\Xi}(t,\cdot)\|_{L^2}|P_{v_2}|-\left\|\big(v-\ml{L}_{\Xi}P_{v_2}\big)(t,\cdot)\right\|_{L^2}\\
	&\gtrsim\|\chi_{\intt}(D)\ml{L}_{\Xi}(t,\cdot)\|_{L^2}|P_{v_2}|-o(t^{-\frac{3n-8\sigma}{4\sigma}})
\end{align*}
for large-time $t\gg1$. Recalling the representations of $\ml{L}_{\Xi}(t,x)$ for $\Xi\in \{(1,3),[3],(3,\infty)\}$, we found that the lower bounds have been derived in Lemma \ref{Lem-03} such that
\begin{align*}
	\|\chi_{\intt}(\xi)\widehat{\ml{L}}_{\Xi}(t,\xi)\|_{L^2}\gtrsim t^{-\frac{3n-8\sigma}{4\sigma}},
\end{align*}
which immediately leads to
\begin{align*}
	\|v(t,\cdot)\|_{L^2}\gtrsim t^{-\frac{3n-8\sigma}{4\sigma}}|P_{v_2}|
\end{align*}
for $t\gg1$. Hence, we complete the sharp estimate for the solution itself. By the same philosophy, the sharpness of the second one in \eqref{Su05} can be proved analogously.
\end{proof}

\begin{remark}
Let us consider the general case $\eta\in(1,\infty)$. We find that $\sigma=\frac{3}{8}n$ is the threshold for decay estimates of the solution itself according to \eqref{Su05}. That is to say optimal polynomial decay when $\sigma\in(0,\frac{3}{8}n)$; sharp bounded estimates when $\sigma=\frac{3}{8}n$; optimal polynomial growth when $\sigma\in(\frac{3}{8}n,\infty)$ for the solution itself in the $L^2$ norm.
\end{remark}


\section{Global (in time) well-posedness for the semilinear Cauchy problem}\label{Sec-GESDS}
\subsection{Philosophy and main tools of our approach}
$\ \ \ \ $For any $T>0$, when \eqref{Restriction(n,sigma)} holds, let us introduce the evolution spaces of fractional order
\begin{align}\label{Evol-Space}
X(T):=\ml{C}\big([0,T],H^{\frac{4}{3}\sigma}\big)\ \ \mbox{with}\ \ \sigma\in(0,\infty),	
\end{align}
carrying its weighted norm 
\begin{align}\label{NEW002}
\|u\|_{X(T)}:=\sup\limits_{t\in[0,T]}\left((1+t)^{\frac{3n-8\sigma}{4\sigma}}\|u(t,\cdot)\|_{L^2}+(1+t)^{\frac{3n}{4\sigma}}\|u(t,\cdot)\|_{\dot{H}^{\frac{4}{3}\sigma}}\right).
\end{align}
Note that the remaining case $n=\frac{8}{3}\sigma$ when $\eta\in(3,\infty)$ does not include in \eqref{Restriction(n,sigma)}, which will be discussed in Remark \ref{Rem-spe}.
 The time-weighted Sobolev norm \eqref{NEW002} is strongly motivated by the sharp $(L^2\cap L^1)-L^2$ type estimates for the linearized Cauchy problem \eqref{Eq-Linear-Third-PDE} with $\eta\in(1,\infty)$, precisely, Propositions \ref{Prop-(1,3)}-\ref{Prop-(3)} and Corollary \ref{Coro-New}. Subsequently, we introduce the operator $\ml{N}$ such that
\begin{align*}
\ml{N}:\ u(t,x)\in X(T)\to \ml{N}u(t,x):=u^{\lin} (t,x)+u^{\non}(t,x),
\end{align*}
where $u^{\lin}(t,x) \equiv v(t,x)$ denotes the solution to the corresponding linearized Cauchy problem \eqref{Eq-Linear-Third-PDE} with the size $\epsilon$ for initial data, and $u^{\non}(t,x)$ is defined via
\begin{align*}
u^{\non}(t,x):=\int_0^tE_2(t-\tau,x)\ast_{(x)}|u(\tau,x)|^p\mathrm{d}\tau,
\end{align*}
where $E_2=E_2(t,x)$ is the fundamental solution to the linearized Cauchy problem \eqref{Eq-Linear-Third-PDE} with initial data $v_0(x)=v_1(x)=0$ and $v_2(x)=\delta_0$, in which $\delta_0$ is the Dirac distribution at $x=0$ with respect to the spatial variables. Particularly,
\begin{align*}
\ml{F}_{x\to\xi}\big(E_2(t,x)\big)=\widehat{K}_2(t,|\xi|):=\begin{cases}
\displaystyle{\sum\limits_{j=1,2,3}\frac{\mathrm{e}^{\lambda_jt}}{\prod\limits_{k=1,2,3,\ k\neq j}(\lambda_j-\lambda_k)}}&\mbox{when}\ \ \eta\neq3,\\[1.6em]
\displaystyle{\frac{1}{2}t^2\exp(-|\xi|^{\frac{2}{3}\sigma}t)}&\mbox{when}\ \ \eta=3,
\end{cases}
\end{align*}
with the characteristic roots \eqref{Characteristic_Roots}, which are well-studied in Section \ref{Sec-linear}. When \eqref{Restriction(n,sigma)} holds, a brief summary of Propositions \ref{Prop-(1,3)}-\ref{Prop-(3)} and Corollary \ref{Coro-New} indicates estimates as follows:
\begin{align}
\|E_2(t,\cdot)\ast_{(x)}g(t,\cdot)\|_{L^2}&\lesssim(1+t)^{-\frac{3n-8\sigma}{4\sigma}}\|g(t,\cdot)\|_{L^2\cap L^1},\label{D1}\\
\|E_2(t,\cdot)\ast_{(x)}g(t,\cdot)\|_{\dot{H}^{\frac{4}{3}\sigma}}&\lesssim(1+t)^{-\frac{3n}{4\sigma}}\|g(t,\cdot)\|_{L^2\cap L^1},\label{D2}\\
\|E_2(t,\cdot)\ast_{(x)}g(t,\cdot)\|_{\dot{H}^{\frac{4}{3}\sigma}}&\lesssim\|g(t,\cdot)\|_{L^2},\label{D3}
\end{align}
for $\eta\in(1,\infty)$ and dimensions $n$ fulfilling \eqref{Restriction(n,sigma)}.

We will demonstrate global (in time) existence and uniqueness of small data Sobolev solutions to the semilinear third-order (in time) evolution equations \eqref{Eq-Third-PDE} with $\eta\in(1,\infty)$ by proving a unique fixed point of the operator $\ml{N}$ that says $\ml{N}u\in X(T)$ for all  $T>0$. Accurately, we will rigorously prove the next two vital estimates uniformly with respect to $T$:
\begin{align}
\|\ml{N}u\|_{X(T)}&\lesssim\epsilon\|(u_0,u_1,u_2)\|_{\ml{B}_{\sigma}}+\|u\|_{X(T)}^p,\label{Cruc-01}\\
\|\ml{N}u-\ml{N}\bar{u}\|_{X(T)}&\lesssim\|u-\bar{u}\|_{X(T)}\big(\|u\|_{X(T)}^{p-1}+\|\bar{u}\|_{X(T)}^{p-1}\big),\label{Cruc-02}
\end{align}
respectively, under some conditions of the power $p$, with the initial data space $\ml{B}_{\sigma}$ defined in \eqref{Data-space-D}. In our desired estimate \eqref{Cruc-02}, $u$ and $\bar{u}$ are two solutions to the semilinear Cauchy problem \eqref{Eq-Third-PDE} with $\eta\in(1,\infty)$. Thus, by assuming $\|(u_0,u_1,u_2)\|_{\ml{B}_{\sigma} }\leqslant C$ and a small parameter $0<\epsilon\ll 1$, we may combine \eqref{Cruc-01} with \eqref{Cruc-02} to declare that there exists a global (in time) small data Sobolev solution $u^{*}=u^{*}(t,x)\in X(\infty)$ by using Banach's fixed point theorem.

To end this subsection, we recall the fractional Gagliardo-Nirenberg inequality \cite{Hajaiej-Molinet-Ozawa-Wang-2011} that will be applied soon afterwards.
\begin{lemma}\label{fractionalgagliardonirenbergineq}
	Let $p,p_0,p_1\in(1,\infty)$ and $\kappa\in[0,s)$ with $s\in(0,\infty)$. Then, for all $f\in L^{p_0}\cap \dot{H}^{s}_{p_1}$ the following interpolation holds:
	\begin{equation*}
	\|g\|_{\dot{H}^{\kappa}_{p}}\lesssim\|g\|_{L^{p_0}}^{1-\beta}\|g\|^{\beta}_{\dot{H}^{s}_{p_1}},
	\end{equation*}
	where $\beta=(\frac{1}{p_0}-\frac{1}{p}+\frac{\kappa}{n})\big/(\frac{1}{p_0}-\frac{1}{p_1}+\frac{s}{n})$ and $\beta\in[\frac{\kappa}{s},1]$.
\end{lemma}

\subsection{Proof of Theorem \ref{Thm-GESDS}: Global (in time) existence of Sobolev solution}\label{Subsec-GESDS}
$\ \ \ \ $First of all, from the third statement of Proposition \ref{Prop-Stable} and some estimates in Propositions \ref{Prop-(1,3)}-\ref{Prop-(3)} as well as Corollary \ref{Coro-New}, we may claim $u^{\lin}\in X(T)$ such that
\begin{align}\label{E1}
\|u^{\lin}\|_{X(T)}\lesssim \epsilon\|(u_0,u_1,u_2)\|_{\ml{B}_{\sigma}}.
\end{align}
Due to $\ml{N}u=u^{\lin}+u^{\non}$, our first goal is to demonstrate 
\begin{align*}
\|u^{\non}\|_{X(T)}\lesssim\|u\|^p_{X(T)}\ \ \mbox{uniformly}\ \ T,
\end{align*}
 under some conditions for the exponent $p$. To do this, we now need to estimate the nonlinearity $|u(\tau,\cdot)|^p$ in the $L^1$ and $L^2$ norms, respectively. Let us apply the fractional Gagliardo-Nirenberg inequality (see Lemma \ref{fractionalgagliardonirenbergineq}) to arrive at
\begin{align*}
\|\,|u(\tau,\cdot)|^p\|_{L^1}&\lesssim \|u(\tau,\cdot)\|_{L^2}^{(1-\beta_1)p}\|u(\tau,\cdot)\|_{\dot{H}^{\frac{4}{3}\sigma}}^{\beta_1 p}\lesssim(1+\tau)^{-\frac{3n-4\sigma}{2\sigma}p+\frac{3n}{2\sigma}}\|u\|_{X(T)}^p
\end{align*}
for $\tau\in[0,T]$ with $\beta_1:=\frac{3n}{4\sigma}(\frac{1}{2}-\frac{1}{p})\in[0,1]$, that is $2\leqslant p\leqslant\frac{6n}{(3n-8\sigma)_+}$. Analogously, 
\begin{align}\label{Nonlinearity-02}
\|\,|u(\tau,\cdot)|^p\|_{L^2}&\lesssim \|u(\tau,\cdot)\|_{L^2}^{(1-\beta_2)p}\|u(\tau,\cdot)\|_{\dot{H}^{\frac{4}{3}\sigma}}^{\beta_2 p}\lesssim(1+\tau)^{-\frac{3n-4\sigma}{2\sigma}p+\frac{3n}{4\sigma}}\|u\|_{X(T)}^p
\end{align}
for $\tau\in[0,T]$ with $\beta_2:=\frac{3n}{4\sigma}(\frac{1}{2}-\frac{1}{2p})\in[0,1]$, that is $1<p\leqslant\frac{3n}{(3n-8\sigma)_+}$. For these reasons, it shows
\begin{align}\label{Nonlinearity-01}
\|\,|u(\tau,\cdot)|^p\|_{L^2\cap L^1}\lesssim(1+\tau)^{-\frac{3n-4\sigma}{2\sigma}p+\frac{3n}{2\sigma}}\|u\|_{X(T)}^p,
\end{align}
where we restricted
\begin{align}\label{Cond-01}
2\leqslant p\leqslant\frac{3n}{(3n-8\sigma)_+}\ \ \mbox{if}\ \ 1\leqslant n\leqslant \frac{16}{3}\sigma.
\end{align}
The restriction \eqref{Cond-01} originates from some applications of the Gagliardo-Nirenberg inequality technically.
In order to estimate the solution itself in the $L^2$ norm, we apply the derived $(L^2\cap L^1)-L^2$ estimate \eqref{D1} in $[0,t]$ as well as the estimate for power nonlinearity \eqref{Nonlinearity-01} to get
\begin{align*}
&(1+t)^{\frac{3n-8\sigma}{4\sigma}}\|u^{\non}(t,\cdot)\|_{L^2}\lesssim(1+t)^{\frac{3n-8\sigma}{4\sigma}} \int_0^t(1+t-\tau)^{-\frac{3n-8\sigma}{4\sigma}}(1+\tau)^{-\frac{3n-4\sigma}{2\sigma}p+\frac{3n}{2\sigma}}\mathrm{d}\tau\,\|u\|_{X(T)}^p\\
&\qquad\lesssim \int_0^{t/2}(1+\tau)^{-\frac{3n-4\sigma}{2\sigma}p+\frac{3n}{2\sigma}}\mathrm{d}\tau\,\|u\|_{X(T)}^p+(1+t)^{\frac{9n-8\sigma}{4\sigma}-\frac{3n-4\sigma}{2\sigma}p}\int_{t/2}^t(1+t-\tau)^{-\frac{3n-8\sigma}{4\sigma}}\mathrm{d}\tau\,\|u\|_{X(T)}^p.
\end{align*}
From the assumption
\begin{align}\label{Cond-02}
p>p_{\mathrm{crit}}(n,\sigma)=1+\frac{6\sigma}{3n-4\sigma}\ \ \mbox{with}\ \ \frac{4}{3}\sigma<n<4\sigma,
\end{align}
one derives
\begin{align*}
-\frac{3n-4\sigma}{2\sigma}p+\frac{3n}{2\sigma}<-1\ \ \mbox{and}\ \ -\frac{3n-8\sigma}{4\sigma}>-1,
\end{align*}
which lead to
\begin{align}\label{E2}
(1+t)^{\frac{3n-8\sigma}{4\sigma}}\|u^{\non}(t,\cdot)\|_{L^2}\lesssim\|u\|_{X(T)}^p.
\end{align}
Associating the conditions \eqref{Cond-01} with \eqref{Cond-02}, we have
\begin{align}\label{Cond-03}
p_{\mathrm{crit}}(n,\sigma)=\max\big\{2,p_{\mathrm{crit}}(n,\sigma) \big\}<p\leqslant\frac{3n}{(3n-8\sigma)_+}
\end{align}
by taking $\frac{4}{3}\sigma<n\leqslant\frac{10}{3}\sigma$.
 Note that the dimension should fulfill \eqref{Restriction(n,sigma)}.   Next, we apply the derived $(L^2\cap L^1)-\dot{H}^{\frac{4}{3}\sigma}$ estimate \eqref{D2} in $[0,t/2]$ associated with \eqref{Nonlinearity-01}, and $L^2-\dot{H}^{\frac{4}{3}\sigma}$ estimate \eqref{D3} in $[t/2,t]$ associated with \eqref{Nonlinearity-02} to obtain
 \begin{align}
 	(1+t)^{\frac{3n}{4\sigma}}\|u^{\non}(t,\cdot)\|_{\dot{H}^{\frac{4}{3}\sigma}}&\lesssim\int_0^{t/2}(1+\tau)^{-\frac{3n-4\sigma}{2\sigma}p+\frac{3n}{2\sigma}}\mathrm{d}\tau\,\|u\|_{X(T)}^p+(1+t)^{\frac{3n}{2\sigma}-\frac{3n-4\sigma}{2\sigma}p}\int_{t/2}^t\mathrm{d}\tau \,\|u\|_{X(T)}^p\notag\\
 	&\lesssim\|u\|_{X(T)}^p\label{E3}
 \end{align}
 under the condition \eqref{Cond-03}. Thus, the combination of \eqref{E1}, \eqref{E2} as well as \eqref{E3} immediately completes
 \begin{align*}
 \|u^{\lin}\|_{X(T)}+\|u^{\non}\|_{X(T)}\lesssim \epsilon\|(u_0,u_1,u_2)\|_{\ml{B}_{\sigma}}+ \|u\|_{X(T)}^p,
 \end{align*}
namely,  the derivation of the crucial estimate \eqref{Cruc-01}.

With the aim of proving \eqref{Cruc-02}, we notice that
\begin{align*}
	\|\ml{N}u-\ml{N}\bar{u}\|_{X(T)}=\Big\| \int_0^tE_2(t-\tau,\cdot)\ast_{(x)}\big(|u(\tau,\cdot)|^p-|\bar{u}(\tau,\cdot)|^p\big)\,\mathrm{d}\tau \Big\|_{X(T)}.
\end{align*}
Thanks to H\"older's inequality, one gets
\begin{align*}
	\|\,|u(\tau,\cdot)|^p-|\bar{u}(\tau,\cdot)|^p\|_{L^m}\lesssim\|u(\tau,\cdot)-\bar{u}(\tau,\cdot)\|_{L^{mp}}\big(\|u(\tau,\cdot)\|^{p-1}_{L^{mp}}+\|\bar{u}(\tau,\cdot)\|^{p-1}_{L^{mp}}\big),
\end{align*}
with $m=1,2$. Finally, employing the fractional Gagliardo-Nirenberg inequality to estimate three terms on the right-hand side of the previous inequality, we are able to derive our desired estimate \eqref{Cruc-02}. Then, our proof is finished.

 \begin{remark}\label{Rem-Scrop-02}
Concerning another case
\begin{align*}
p>\max\left\{\frac{9n-8\sigma}{2(3n-4\sigma)_+},p_{\mathrm{crit}}(n,\sigma) \right\}\ \ \mbox{and}\ \ 2\leqslant p\leqslant\frac{3n}{(3n-8\sigma)_+},
\end{align*}
we can also prove \eqref{E2} for $4\sigma\leqslant n\leqslant\frac{16}{3}\sigma$ because
\begin{align*}
	(1+t)^{\frac{9n-8\sigma}{4\sigma}-\frac{3n-4\sigma}{2\sigma}p}\int_{t/2}^t(1+t-\tau)^{-\frac{3n-8\sigma}{4\sigma}}\mathrm{d}\tau\lesssim 1.
\end{align*}
However, it does not imply the lower bound with the critical exponent $p_{\mathrm{crit}}(n,\sigma)$, and the general global (in time) existence result is beyond the scope of the present paper.
 \end{remark}
\begin{remark}\label{Rem-spe}
For $n=\frac{8}{3}\sigma$ when $\eta\in(3,\infty)$, we just need to modify the norm \eqref{NEW002} by
\begin{align*}
\|u\|_{X(T)}:=\sup\limits_{t\in[0,T]}\left([\ln(\mathrm{e}+t)]^{-1}\|u(t,\cdot)\|_{L^2}+(1+t)^2\|u(t,\cdot)\|_{\dot{H}^{\frac{4}{3}\sigma}}\right)
\end{align*}
motivated from Corollary \ref{Coro-New}. At this time, the estimate \eqref{D1} will be changed into
\begin{align*}
	\|E_2(t,\cdot)\ast_{(x)}g(t,\cdot)\|_{L^2}\lesssim \ln (\mathrm{e}+t)\|g(t,\cdot)\|_{L^2\cap L^1}
\end{align*}
for $n=\frac{8}{3}\sigma$ when $\eta\in(3,\infty)$. Again, by applying Proposition \ref{fractionalgagliardonirenbergineq}, we arrive at
\begin{align*}
\|\,|u(\tau,\cdot)|^p\|_{L^2\cap L^1}&\lesssim [\ln(\mathrm{e}+\tau)]^2(1+\tau)^{-2(p-2)}\|u\|_{X(T)}^p,\\
\|\,|u(\tau,\cdot)|^p\|_{L^2}&\lesssim [\ln(\mathrm{e}+\tau)](1+\tau)^{-2(p-1)}\|u\|_{X(T)}^p,
\end{align*}
under the restriction \eqref{Cond-01}. Then,  with the same procedure as the above one, we also can demonstrate \eqref{Cruc-01} provided that the exponent $p$ satisfies \eqref{Cond-03}.
\end{remark}

\begin{remark}\label{Rem-Further_regularity}
Providing that one is interested in the regularities of global (in time) solutions $u$, under the same conditions as those in Propositions \ref{Prop-(1,3)}-\ref{Prop-(3)} and Corollary \ref{Coro-New}, if \eqref{Restriction(n,sigma)} holds, the following estimates for the linearized Cauchy problem \eqref{Eq-Linear-Third-PDE} can be proved easily:
\begin{align*}
	\|\partial_t^jv(t,\cdot)\|_{L^2}&\lesssim(1+t)^{-j-\frac{3n-8\sigma}{4\sigma}}\|(v_0,v_1,v_2)\|_{\ml{B}_{\sigma}},\\
	\|\partial_t^jv(t,\cdot)\|_{\dot{H}^{\frac{2}{3}\sigma(2-j)}}&\lesssim(1+t)^{-\frac{3n}{4\sigma}}\|(v_0,v_1,v_2)\|_{\ml{B}_{\sigma}},
\end{align*}
with $j=0,1,2$, which is a generalization of Propositions \ref{Prop-(1,3)}-\ref{Prop-(3)}. Hence, we may introduce the solution space
\begin{align*}
	Y(T):=\ml{C}\big([0,T],H^{\frac{4}{3}\sigma}\big)\cap \ml{C}^1\big([0,T],H^{\frac{2}{3}\sigma}\big)\cap\ml{C}^2\big([0,T],L^2\big)
\end{align*}
for $T>0$ and any $\sigma\in(0,\infty)$, equipped the time-weighted Sobolev norm
\begin{align*}
	\|u\|_{Y(T)}:=\sup\limits_{t\in[0,T]}\left(\,\sum\limits_{j=0,1,2}(1+t)^{j+\frac{3n-8\sigma}{4\sigma}}\|\partial_t^ju(t,\cdot)\|_{L^2}+(1+t)^{\frac{3n}{4\sigma}}\sum\limits_{j=0,1}\|\partial_t^ju(t,\cdot)\|_{\dot{H}^{\frac{2}{3}\sigma(2-j)}}\right).
\end{align*}
By repeating some analogous procedure as we have treated $\|u^{\non}(t,\cdot)\|_{L^2}$ and $\|u^{\non}(t,\cdot)\|_{\dot{H}^{\frac{4}{3}\sigma}}$ in Subsection \ref{Subsec-GESDS}, we may demonstrate \eqref{Cruc-01}-\eqref{Cruc-02} in  $Y(T)$ instead of $X(T)$. Namely, by lengthy but straightforward computations, under the same condition of the exponent $p$ and regularities for small initial data, there is a uniquely determined Sobolev solution $u\in Y(\infty)$ to the third-order (in time) evolution equations \eqref{Eq-Third-PDE} with $\eta\in(1,\infty)$. Furthermore, the solution satisfies the following estimates:
\begin{align*}
	\|\partial_t^ju(t,\cdot)\|_{L^2}&\lesssim \epsilon(1+t)^{-j-\frac{3n-8\sigma}{4\sigma}}\|(u_0,u_1,u_2)\|_{\ml{B}_{\sigma}},\\
	\|\partial_t^ju(t,\cdot)\|_{\dot{H}^{\frac{2}{3}\sigma(2-j)}}&\lesssim \epsilon(1+t)^{-\frac{3n}{4\sigma}}\|(u_0,u_1,u_2)\|_{\ml{B}_{\sigma}},
\end{align*}
with $j=0,1,2$.
\end{remark}

\section{Blow-up for  the semilinear Cauchy problem}\label{Sec-Blow-up}
\subsection{Preliminary with fractional Laplacian}
$\ \ \ \ $Before proving blow-up of weak solutions to the semilinear problem \eqref{Eq-Third-PDE} with fractional Laplacian, let us recall some auxiliary lemmas derived in \cite{Dao-Reissig=2020,D'Abbicco-Fujiwara=2021}.
\begin{lemma} \label{lemma2.1}
Let $q>n$ and $\sigma\in(0,\infty)$. Setting $\tilde{\sigma}=\sigma-[\sigma]$, then it holds
\begin{align*}
|(-\Delta)^{\sigma}\langle x\rangle^{-q}|\lesssim\langle x\rangle^{-q_{\sigma}}\ \ \mbox{for}\ \ x\in\mb{R}^n,
\end{align*}
where $q_{\sigma}=q+2\sigma$ if $\sigma$ is an integer, or $q_{\sigma}=n+2\tilde{\sigma}$, otherwise.
\end{lemma}
\begin{lemma} \label{lemma2.2}
 Setting $\phi(x):= \langle x\rangle^{-q}$ and $\phi_R(x):= \phi(R^{-1}x)$ for any $R>0$, then it holds
	\begin{equation*}
		(-\Delta)^\gamma (\phi_R)(x)= R^{-2\gamma} \big((-\Delta)^\gamma \phi \big)(R^{-1}x)
	\end{equation*}
for any number $\gamma\geqslant 0$.
\end{lemma}

Let us define some test functions with suitable scale that strongly relies on the new adjoint operator of \eqref{Eq-Third-PDE}, i.e.
\begin{align*}
-\partial_t^3+\ml{A}+\eta\ml{A}^{\frac{1}{3}}\partial_t^2-\eta\ml{A}^{\frac{2}{3}}\partial_t.
\end{align*}
 For one thing, we set $\varphi(x):=\langle x\rangle^{-n-2s_{\sigma}}$, where $s_{\sigma}\in(0,1)$ is chosen as an arbitrary constant if $\sigma\in\mb{N}$, and as a small constant satisfying $s_{\sigma}\in(0,\sigma-[\sigma])$ if $\sigma\in\mb{R}_+\backslash\mb{N}$.  For another, we introduce $\zeta\in\ml{C}^{\infty}$ with its support $\zeta(t)\subset[0,1]$ such that $\zeta(t)=1$ if $t\in[0,\frac{1}{2}]$, non-increasing, and $\zeta(t)=0$ if $t\in[1,\infty)$. Eventually, we put new test functions
\begin{align*}
	\varphi_R(x):=\langle R^{-1}K^{-1}x\rangle^{-n-2s_{\sigma}}\ \ \mbox{and}\ \ \zeta_R(t):=\zeta(R^{-\frac{2}{3}\sigma}t)
\end{align*}
with a suitable constant $K\geqslant 1$ and a sufficiently large parameter $R\gg1$.
\subsection{Proof of Theorem \ref{Thm-Blow-up}: Blow-up of weak solutions}
$\ \ \ \ $To begin with the proof, according to the nonlinearity, let us introduce
\begin{align*}
	I_{R}:=\int_0^{\infty}\int_{\mb{R}^n}|u(t,x)|^p\varphi_R(x)\zeta_R(t)\mathrm{d}x\mathrm{d}t.
\end{align*}
Assume by contradiction that $u=u(t,x)$ is a global (in time) weak solution to the semilinear Cauchy problem \eqref{Eq-Third-PDE} with a general parameter $\eta\in(0,\infty)$. Thus, it holds that
\begin{align}
	&\epsilon\int_{\mb{R}^n}\left(\eta u_0(x)\ml{A}^{\frac{2}{3}}\varphi_R(x)+\eta u_1(x)\ml{A}^{\frac{1}{3}}\varphi_R(x)+u_2(x)\varphi_R(x)\right)\mathrm{d}x+I_R\notag\\
	&\qquad=\int_0^{\infty}\int_{\mb{R}^n}u(t,x)\left(-\partial_t^3+\ml{A}+\eta\ml{A}^{\frac{1}{3}}\partial_t^2-\eta\ml{A}^{\frac{2}{3}}\partial_t\right)\big(\varphi_R(x)\zeta_R(t)\big)\mathrm{d}x\mathrm{d}t,\label{Weak_Solution}
\end{align}
carrying $\ml{A}=(-\Delta)^{\sigma}$ with $\sigma\in(0,\infty)$, that is, the solution $u$ has such properties that all the integrals are well-defined for $R>0$. In the first integral of \eqref{Weak_Solution}, the operator $\ml{A}^{\alpha}$ acting on the test function $\varphi_R(x)$ only occurs for the initial data $u_0(x),u_1(x)$, which leads to $R^{-\frac{4}{3}\sigma}$ as well as $R^{-\frac{2}{3}\sigma}$, respectively. It avoids sign conditions for $u_0(x)$ and $u_1(x)$ later by taking $R\gg1$.

Applying H\"older's inequality, and taking $K=1$, we may estimate
\begin{align}
&\epsilon\int_{\mb{R}^n}\left(\eta u_0(x)\ml{A}^{\frac{2}{3}}\varphi_R(x)+\eta u_1(x)\ml{A}^{\frac{1}{3}}\varphi_R(x)+u_2(x)\varphi_R(x)\right)\mathrm{d}x+I_R\notag\\
&\qquad\lesssim\int_0^{\infty}\int_{\mb{R}^n}\left|u(t,x)\left(-\partial_t^3+\ml{A}+\eta\ml{A}^{\frac{1}{3}}\partial_t^2-\eta\ml{A}^{\frac{2}{3}}\partial_t\right)\big(\varphi_R(x)\zeta_R(t)\big)\right|\mathrm{d}x\mathrm{d}t\notag\\
&\qquad\lesssim R^{-2\sigma}\int_0^{R^{\frac{2}{3}\sigma}}\int_{\mb{R}^n}|u(t,x)|\varphi_R(x)\big(|\zeta_R'''(t)|+|\zeta_R(t)|+|\zeta_R''(t)|+|\zeta_R'(t)|\big)\mathrm{d}x\mathrm{d}t\notag\\
&\qquad\lesssim R^{-2\sigma}I_R^{\frac{1}{p}}\left(R^{n+\frac{2}{3}\sigma}\int_{\mb{R}^n}\langle\tilde{x}\rangle^{-n-2s_{\sigma}}\mathrm{d}\tilde{x}\right)^{\frac{1}{p'}}\notag\\
&\qquad\lesssim R^{\frac{3n+2\sigma}{3p'}-2\sigma}I_R^{\frac{1}{p}},\label{Blow-up-01}
\end{align}
where we employed Lemmas \ref{lemma2.1} and \ref{lemma2.2} to treat the nonlocal terms. From the condition
\begin{align*}
1<p<p_{\mathrm{crit}}(n,\sigma)=\frac{3n+2\sigma}{(3n-4\sigma)_+}\ \ \mbox{that is}\ \ \frac{1}{p'}<\frac{6\sigma}{3n+2\sigma},
\end{align*}
we know $R^{\frac{3n+2\sigma}{3}-2\sigma p'}\to0$ while $R\to\infty$. Then, recalling the crucial estimate \eqref{Blow-up-01} with Young's inequality, we state the contradiction as follows:
\begin{align}\label{E4} 
0\leqslant \lim\limits_{R\to\infty}I_R\lesssim \lim\limits_{R\to\infty}\left(R^{\frac{3n+2\sigma}{3}-2\sigma p'}-\epsilon\int_{\mb{R}^n}u_2(x)\varphi_R(x)\mathrm{d}x\right)<0,
\end{align}
thanks to the sign assumption \eqref{Sign_Assumption}. Therefore, it implies that any non-trivial weak solution may blow up in finite time. To more accurately describe the blow-up time for this weak solution, we take 
\begin{align*}
R\uparrow T_{\epsilon,\mathrm{w}}^{\frac{3}{2\sigma}}\ \ \mbox{with}\ \ \sigma\in(0,\infty)	
\end{align*}
 in \eqref{E4}. Consequently, a contradiction  immediately occurs provided that
 \begin{align*}
 T_{\epsilon,\mathrm{w}}^{\frac{3}{2\sigma}\left(\frac{3n+2\sigma}{3}-2\sigma p'\right)}\leqslant  C_{u_2}\epsilon,
 \end{align*}
which leads to 
\begin{align}\label{Upper-Subcrit}
	T_{\epsilon,\mathrm{w}}\leqslant C\epsilon^{-\frac{2\sigma}{6\sigma p'-(3n+2\sigma)}}\ \ \mbox{if}\ \ 1<p<p_{\mathrm{crit}}(n,\sigma).
\end{align}
It completes the proof of upper bound estimates for the lifespan in the subcritical case, namely, the first subcase of Theorem \ref{Thm-Lifespan-upper}.

Turning to the critical case $p=p_{\mathrm{crit}}(n,\sigma)$, we first deduce that $I_R$ is uniformly bounded, namely, $u\in L^p([0,\infty)\times\mb{R}^n)$. Moreover, from \eqref{Blow-up-01} associated with Young's inequality, we claim
\begin{align*}
\lim\limits_{R\to\infty}\int_0^{\infty}\int_{\mb{R}^n}u(t,x)\left(-\partial_t^3+\eta\ml{A}^{\frac{1}{3}}\partial_t^2-\eta\ml{A}^{\frac{2}{3}}\partial_t\right)\big(\varphi_R(x)\zeta_R(t)\big)\mathrm{d}x\mathrm{d}t=0
\end{align*}
because of $\zeta_R(t)\equiv1$ when $t\in[0,\frac{1}{2}R^{\frac{2}{3}\sigma}]$. It follows that
\begin{align*}
	0\leqslant \lim\limits_{R\to\infty}I_R&=\lim\limits_{R\to\infty}\int_0^{\infty}\int_{\mb{R}^n}u(t,x)\left(-\partial_t^3+\eta\ml{A}^{\frac{1}{3}}\partial_t^2-\eta\ml{A}^{\frac{2}{3}}\partial_t\right)\big(\varphi_R(x)\zeta_R(t)\big)\mathrm{d}x\mathrm{d}t\\
	&\quad+\lim\limits_{R\to\infty}\int_0^{\infty}\int_{\mb{R}^n}u(t,x)\ml{A}\varphi_R(x)\zeta_R(t)\mathrm{d}x\mathrm{d}t-\epsilon\lim\limits_{R\to\infty}\int_{\mb{R}^n}u_2(x)\varphi_R(x)\mathrm{d}x \\
	&\quad-\epsilon\lim\limits_{R\to\infty}\int_{\mb{R}^n}\left(\eta u_0(x)\ml{A}^{\frac{2}{3}}\varphi_R(x)+\eta u_1(x)\ml{A}^{\frac{1}{3}}\varphi_R(x)\right)\mathrm{d}x\\
	&\lesssim K^{n-2\sigma p'}-\epsilon\lim\limits_{R\to\infty}\int_{\mb{R}^n}u_2(x)\varphi_R(x)\mathrm{d}x,
\end{align*}
in which the term with $K$ comes from the nonlocal term without time-derivative $\ml{A}\varphi_R(x)$ thanks to Lemma \ref{lemma2.2}. Finally, the contradiction occurs for taking sufficient large $K\gg1$, due to \eqref{Sign_Assumption} and $n-2\sigma p'=-\frac{2}{3}\sigma<0$ in the critical case $p=p_{\mathrm{crit}}(n,\sigma)$. We postpone the lifespan estimates from the above side in the next section due to its challenge. Our proof on blow-up of solutions is complete.

\section{Sharp lifespan estimates for the semilinear Cauchy problem}\label{Sec-Lifespan}
\subsection{Proof of Theorem \ref{Thm-Lifespan-lower}: Lower bound estimates for the lifespan}\label{Subsec-Lower}
$\ \ \ \ $We will use some notations from Section \ref{Sec-GESDS}, particularly, the evolution space  $X(T)$ introduced in \eqref{Evol-Space}. In order to derive lower bound estimates for the lifespan, we will deduce a new nonlinear inequality with time-dependent coefficients instead of \eqref{Cruc-01} because of $1<p\leqslant p_{\mathrm{crit}}(n,\sigma)$. Let us consider the dimension fulfills the condition \eqref{Restriction(n,sigma)}, and the remaining case for $n=\frac{8}{3}\sigma$ when $\eta\in(3,\infty)$ can be treated similarly by following Remark \ref{Rem-spe}. 

Thanks to the definition of mild solutions in \eqref{Mild-Solution}, we may represent and estimate local (in time) mild solutions to \eqref{Eq-Third-PDE} straightforwardly. According to the derived estimates \eqref{Nonlinearity-01} in Section \ref{Sec-GESDS}, one arrives at
\begin{align*}
(1+t)^{\frac{3n-8\sigma}{4\sigma}}\|u(t,\cdot)\|_{L^2}&\lesssim\epsilon\|(u_0,u_1,u_2)\|_{\ml{B}_{\sigma}} +\int_0^{t/2}(1+\tau)^{-\frac{3n-4\sigma}{2\sigma}p+\frac{3n}{2\sigma}}\mathrm{d}\tau\,\|u\|_{X(T)}^p\\
&\quad+(1+t)^{1+\frac{3n}{2\sigma}-\frac{3n-4\sigma}{2\sigma}p}\|u\|_{X(T)}^p,
\end{align*}
due to $n<4\sigma$ and the restriction \eqref{Cond-01} from some applications of the fractional Gagliardo-Nirenberg inequality. Moreover, the condition \eqref{Restriction(n,sigma)} for the sharp estimates of solutions to the linearized Cauchy problem should be guaranteed. Let us denote the growth function
\begin{align*}
L(t):=\begin{cases}
(1+t)^{1+\frac{3n}{2\sigma}-\frac{3n-4\sigma}{2\sigma}p}&\mbox{if}\ \ 1<p<p_{\mathrm{crit}}(n,\sigma),\\
\ln(\mathrm{e}+t)&\mbox{if}\ \ p=p_{\mathrm{crit}}(n,\sigma).
\end{cases}
\end{align*}
Since $-\frac{3n-4\sigma}{2\sigma}p+\frac{3n}{2\sigma}=-1$ if $p=p_{\mathrm{crit}}(n,\sigma)$, we can obtain
\begin{align*}
(1+t)^{\frac{3n-8\sigma}{4\sigma}}\|u(t,\cdot)\|_{L^2}\lesssim\epsilon\|(u_0,u_1,u_2)\|_{\ml{B}_{\sigma}} +L(t)\|u\|_{X(T)}^p,
\end{align*}
and, analogously,
\begin{align*}
(1+t)^{\frac{3n}{4\sigma}}\|u(t,\cdot)\|_{\dot{H}^{\frac{4}{3}\sigma}}\lesssim\epsilon\|(u_0,u_1,u_2)\|_{\ml{B}_{\sigma}} +L(t)\|u\|_{X(T)}^p.
\end{align*}
Recalling the definition of $X(T)$ in \eqref{Evol-Space}, the next crucial nonlinear inequality holds:
\begin{align}\label{Cruc-03}
\|u\|_{X(T)}\leqslant C_0\epsilon+C_1L(t)\|u\|_{X(T)}^p,
\end{align}
where $C_0,C_1$ are two positive constants independent of $\epsilon,T$, providing that
\begin{align*}
2\leqslant p\leqslant\min\left\{p_{\mathrm{crit}}(n,\sigma),\frac{3n}{(3n-8\sigma)_+} \right\}=p_{\mathrm{crit}}(n,\sigma).
\end{align*}
The above minimal competition trivially holds when $n\leqslant\frac{8}{3}\sigma$, and
\begin{align*}
1+\frac{6\sigma}{3n-4\sigma}=p_{\mathrm{crit}}(n,\sigma)\leqslant\frac{3n}{3n-8\sigma}\ \ \mbox{when}\ \ \frac{8}{3}\sigma<n\leqslant\frac{10}{3}\sigma.
\end{align*}
 Note that from Section \ref{Sec-GESDS}, we restricted the uniform estimate $L(t)\lesssim 1$ in the supercritical case $p>p_{\mathrm{crit}}(n,\sigma)$ to ensure global (in time) existence of solutions.

Let us now introduce
\begin{align*}
T^*:=\sup\big\{T\in[0,T_{\epsilon,\mathrm{m}})\ \ \mbox{such that}\ \ \ml{G}(T):=\|u\|_{X(T)}\leqslant M\epsilon \big\}
\end{align*}
with a sufficient large constant $M>0$ to be fixed later, which means $\ml{G}(T^*)\leqslant M\epsilon$. Plugging this upper bound into \eqref{Cruc-03} with $T=T^*$, we claim
\begin{align*}
\ml{G}(T^*)\leqslant C_0\epsilon+C_1L(T^*)\ml{G}(T^*)^p\leqslant \epsilon\left(C_0+C_1M^{p}\epsilon^{p-1}L(T^*)\right).
\end{align*}
By choosing a large constant $M$ such that
\begin{align*}
	M>4C_0\ \ \mbox{as well as}\ \ 4C_1M^{p-1}\epsilon^{p-1}L(T^*)<1,
\end{align*}
 we may have $\ml{G}(T^*)<\frac{M}{2}\epsilon$. Due to the fact that $\ml{G}(T)$ is a continuous function for any $T\in(0,T_{\epsilon,\mathrm{m}})$, nevertheless, the last estimate verifies that there exists a time $T_0\in(T^*,T_{\epsilon,\mathrm{m}})$ such that $\ml{G}(T_0)\leqslant M\epsilon$. It contradicts to the definition of $T^*$, i.e. the supremum of time for $\ml{G}(T)\leqslant M\epsilon$. Namely, we need to propose the reverse relation as follows:
\begin{align}\label{E5}
4C_1M^{p-1}\epsilon^{p-1}L(T^*)\geqslant 1.
\end{align}
Consequently, according to \eqref{E5} and the expression of $L(T^*)$ with $T^*\leqslant T_{\epsilon,\mathrm{m}}$, the maximum time of local (in time) existence for mild solutions can be estimated from the below side such that
\begin{itemize}
	\item when $1<p<p_{\mathrm{crit}}(n,\sigma)$, it holds
	\begin{align*}
	T_{\epsilon,\mathrm{m}}\geqslant C\epsilon^{-\frac{p-1}{-\frac{3n-4\sigma}{2\sigma}p+1+\frac{3n}{2\sigma}}}=C\epsilon^{-\frac{2\sigma}{6\sigma p'-(3n+2\sigma)}},
	\end{align*}
	where we used H\"older's conjugate $p'=\frac{p}{p-1}$;
	\item when $p=p_{\mathrm{crit}}(n,\sigma)$, it holds
	\begin{align*}
	T_{\epsilon,\mathrm{m}}\geqslant \exp(C\epsilon^{-(p-1)}).
	\end{align*}
\end{itemize}
In conclusion, we have achieved our aim of lower bound estimates of the lifespan.

\subsection{Proof of Theorem \ref{Thm-Lifespan-upper}: Upper bound estimates for the lifespan}\label{Subsec-Upper}
$\ \ \ \ $Note that the lifespan in the subcritical case $1<p<p_{\mathrm{crit}}(n,\sigma)$ has been estimated in \eqref{Upper-Subcrit}. For this reason, we will concentrate on upper bound estimates for the lifespan $T_{\epsilon,\mathrm{w}}$ in the critical case $p=p_{\mathrm{crit}}(n,\sigma)$ only. Nevertheless, the proof associated with new test functions is quite different from the one of Theorem \ref{Thm-Blow-up}, and upper bound estimates for the lifespan in the critical case is always a delicate part (for example, the semilinear classical damped waves \cite{Lai-Zhou=2019}). Throughout this subsection, we clarify $p=p_{\mathrm{crit}}(n,\sigma)$ in all statements.

Let us define the size of supports for initial data $u_j\in\ml{C}_0^{\infty}$ via
\begin{align*}
r_0:=\max\Big\{|x|:\ x\in \bigcup\limits_{j=0,1,2}\mathrm{supp}\,u_j(x) \Big\}.
\end{align*}
Different from those in the proof of Theorem \ref{Thm-Blow-up}, we now introduce the $\ml{C}_0^{\infty}$ test function $\chi=\chi(\tau)$ and the $L^{\infty}$ cutoff function $\chi^*=\chi^*(\tau)$ as follows:
\begin{align*}
\chi(\tau):=\begin{cases}
1&\mbox{if}\ \ \tau\in[0,\frac{1}{2}],\\
\mbox{decreasing}&\mbox{if}\ \ \tau\in(\frac{1}{2},1),\\
0&\mbox{if}\ \ \tau\in[1,\infty),
\end{cases}
\ \ \mbox{and}\ \ 
\chi^*(\tau):=\begin{cases}
0&\mbox{if}\ \ \tau\in[0,\frac{1}{2}),\\
\chi(\tau)&\mbox{if}\ \ \tau\in[\frac{1}{2},\infty).
\end{cases}
\end{align*}
For a large parameter $R\in(0,\infty)$, we may define
\begin{align}\label{Su01}
\psi_R(t,x):=\left[\chi\left(\tfrac{t+|x|^{\frac{2}{3}\sigma}}{R}\right)\right]^m\ \ \mbox{and}\ \ \psi^*_R(t,x):=\left[\chi^*\left(\tfrac{t+|x|^{\frac{2}{3}\sigma}}{R}\right)\right]^m
\end{align}
with a suitable parameter $m>2\sigma$ that will be chosen later. 

Let us multiply the test function $\psi_R(t,x)$ on the both sides of the equation in \eqref{Eq-Third-PDE}, and then integrate the resultant over $\mb{R}^n$ to deduce
\begin{align}\label{E6}
\int_{\mb{R}^n}|u(t,x)|^p\psi_R(t,x)\mathrm{d}x&=\int_{\mb{R}^n}u(t,x)\left(-\partial_t^3+\ml{A}+\eta\ml{A}^{\frac{1}{3}}\partial_t^2-\eta\ml{A}^{\frac{2}{3}}\partial_t\right)\psi_R(t,x)\mathrm{d}x\notag\\
&\quad+\frac{\mathrm{d}^3}{\mathrm{d}t^3}\int_{\mb{R}^n}u(t,x)\psi_R(t,x)\mathrm{d}x+\eta\frac{\mathrm{d}^2}{\mathrm{d}t^2}\int_{\mb{R}^n}u(t,x)\ml{A}^{\frac{1}{3}}\psi_R(t,x)\mathrm{d}x\notag\\
&\quad+\eta\frac{\mathrm{d}}{\mathrm{d}t}\int_{\mb{R}^n}u(t,x)\ml{A}^{\frac{2}{3}}\psi_R(t,x)\mathrm{d}x-3\frac{\mathrm{d}}{\mathrm{d}t}\int_{\mb{R}^n}u_t(t,x)\psi_{R,t}(t,x)\mathrm{d}x\notag\\
&\quad-2\eta\frac{\mathrm{d}}{\mathrm{d}t}\int_{\mb{R}^n}u(t,x)\ml{A}^{\frac{1}{3}}\psi_{R,t}(t,x)\mathrm{d}x,
\end{align}
in which we denoted $\psi_{R,t}(t,x):=\partial_t\psi_R(t,x)$.  Before constructing a nonlinear differential inequality, we need to estimate time-derivatives and fractional Laplacians acting on the test function from the first term on the right-hand side of \eqref{E6}. A straightforward computation yields
\begin{align*}
\psi_{R,ttt}(t,x)&=m(m-1)(m-2)R^{-3}\left[\chi\left(\tfrac{t+|x|^{\frac{2}{3}\sigma}}{R}\right)\right]^{m-3}\left[\chi'\left(\tfrac{t+|x|^{\frac{2}{3}\sigma}}{R}\right)\right]^3\\
&\quad+3m(m-1)R^{-3}\left[\chi\left(\tfrac{t+|x|^{\frac{2}{3}\sigma}}{R}\right)\right]^{m-2}\chi'\left(\tfrac{t+|x|^{\frac{2}{3}\sigma}}{R}\right)\chi''\left(\tfrac{t+|x|^{\frac{2}{3}\sigma}}{R}\right)\\
&\quad+mR^{-3}\left[\chi\left(\tfrac{t+|x|^{\frac{2}{3}\sigma}}{R}\right)\right]^{m-1}\chi'''\left(\tfrac{t+|x|^{\frac{2}{3}\sigma}}{R}\right).
\end{align*}
From the settings that $\chi\in\ml{C}_0^{\infty}([0,\infty))$, $|\chi(\tau)|\leqslant 1$ and
\begin{align*}
0<\left|\chi^{(k)}\left(\tfrac{t+|x|^{\frac{2}{3}\sigma}}{R}\right)\right|\lesssim 1\ \ \mbox{for}\ \ \frac{R}{2}<t+|x|^{\frac{2}{3}\sigma}<R
\end{align*}
with $k=1,2,3$, we may claim
\begin{align*}
|\psi_{R,ttt}(t,x)|\lesssim R^{-3}\left[\chi^{*}\left(\tfrac{t+|x|^{\frac{2}{3}\sigma}}{R}\right)\right]^{m-3}=R^{-3}[\psi_R^*(t,x)]^{\frac{m-3}{m}}.
\end{align*}
To control the terms $(-\Delta)^{q\sigma}\psi_R(t,x)$ with $q=\frac{1}{3},\frac{2}{3},1$, motivated by \cite[Lemma 3.3]{Dao-Reissig=2019}, we first apply the derivative of composed function with any multi-index $\alpha$ to obtain
\begin{align*}
\left|\partial_x^{\alpha}\chi\left(\tfrac{t+|x|^{\frac{2}{3}\sigma}}{R}\right)\right|&\lesssim\sum\limits_{k=1}^{|\alpha|}\chi^{(k)}\left(\tfrac{t+|x|^{\frac{2}{3}\sigma}}{R}\right)\sum\limits_{|\gamma_1|+\cdots+|\gamma_k|=|\alpha|}\left|\partial_x^{\gamma_1}\left(\tfrac{t+|x|^{\frac{2}{3}\sigma}}{R}\right)\right|\cdots\left|\partial_x^{\gamma_k}\left(\tfrac{t+|x|^{\frac{2}{3}\sigma}}{R}\right)\right|\\
&\lesssim\sum\limits_{k=1}^{|\alpha|}\chi^{(k)}\left(\tfrac{t+|x|^{\frac{2}{3}\sigma}}{R}\right)|x|^{\frac{2}{3}\sigma k-|\alpha|}R^{-k}\\
&\lesssim R^{-1}|x|^{\frac{2}{3}\sigma-|\alpha|},
\end{align*}
where we employed $\tau=t+|x|^{\frac{2}{3}\sigma}\leqslant R$, otherwise $\chi^{(k)}(\tau)=0$ for $k=1,\dots,|\alpha|$. Note that in the last chain, we took $k=1$ due to $|x|^{\frac{2}{3}\sigma}\leqslant R$. From the assumption $\sigma\in3\mb{N}$ implying $q\sigma\in\mb{N}$, a further application of such rule shows
\begin{align*}
|(-\Delta)^{q\sigma}\psi_R(t,x)|&\lesssim\sum\limits_{k=1}^{2q\sigma}\left[\chi\left(\tfrac{t+|x|^{\frac{2}{3}\sigma}}{R}\right)\right]^{m-k}\sum\limits_{|\gamma_1|+\cdots+|\gamma_{k}|=2q\sigma}\left|\partial_x^{\gamma_1}\chi\left(\tfrac{t+|x|^{\frac{2}{3}\sigma}}{R}\right)\right|\cdots\left|\partial_x^{\gamma_k}\chi\left(\tfrac{t+|x|^{\frac{2}{3}\sigma}}{R}\right)\right|\\
&\lesssim\sum\limits_{k=1}^{2q\sigma}\left[\chi^*\left(\tfrac{t+|x|^{\frac{2}{3}\sigma}}{R}\right)\right]^{m-k}|x|^{\frac{2}{3}\sigma k-2q\sigma}R^{-k}\\
&\lesssim R^{-3q}\left[\chi^*\left(\tfrac{t+|x|^{\frac{2}{3}\sigma}}{R}\right)\right]^{m-2q\sigma}=R^{-3q}[\psi_R^*(t,x)]^{\frac{m-2q\sigma}{m}},
\end{align*}
where we employed $\frac{R}{2}\leqslant |x|^{\frac{2}{3}\sigma}\leqslant R$ again in the last line. Therefore, by an analogous way, one has
\begin{align*}
|(-\Delta)^{\sigma}\psi_R(t,x)|&\lesssim R^{-3}[\psi_R^*(t,x)]^{\frac{m-2\sigma}{m}},\\
|(-\Delta)^{\frac{2}{3}\sigma}\psi_{R,t}(t,x)|&\lesssim R^{-3}[\psi_R^*(t,x)]^{\frac{m-\frac{4}{3}\sigma-1}{m}},\\
|(-\Delta)^{\frac{1}{3}\sigma}\psi_{R,tt}(t,x)|&\lesssim R^{-3}[\psi_R^*(t,x)]^{\frac{m-\frac{2}{3}\sigma-2}{m}}.
\end{align*}
The summary of last estimates with $\sigma\in3\mb{N}$ indicates
\begin{align*}
\left|\left(-\partial_t^3+\ml{A}+\eta\ml{A}^{\frac{1}{3}}\partial_t^2-\eta\ml{A}^{\frac{2}{3}}\partial_t\right)\psi_R(t,x)\right|\lesssim R^{-3}[\psi_R^*(t,x)]^{1-\frac{2\sigma }{m}}.
\end{align*}
For this reason, let us integrate \eqref{E6} over $(0,T_{\epsilon,\mathrm{w}})$ to get
\begin{align*}
&\int_0^{T_{\epsilon,\mathrm{w}}}\int_{\mb{R}^n}|u(t,x)|^p\psi_R(t,x)\mathrm{d}x\mathrm{d}t\lesssim R^{-3}\int_0^{T_{\epsilon,\mathrm{w}}}\int_{\mb{R}^n}|u(t,x)|[\psi_R^*(t,x)]^{1-\frac{2\sigma}{m}}\mathrm{d}x\mathrm{d}t\\
&\qquad+\int_{\mb{R}^n}\big(u_{tt}(t,x)\psi_R(t,x)+2u_t(t,x)\psi_{R,t}(t,x)+u(t,x)\psi_{R,tt}(t,x)\big)\mathrm{d}x\Big|_{t=0}^{t=T_{\epsilon,\mathrm{w}}}\\
&\qquad+\int_{\mb{R}^n}\big(\eta u_{t}(t,x)\ml{A}^{\frac{1}{3}}\psi_{R}(t,x)+\eta u(t,x)\ml{A}^{\frac{1}{3}}\psi_{R,t}(t,x)\big)\mathrm{d}x\Big|_{t=0}^{t=T_{\epsilon,\mathrm{w}}}\\
&\qquad+\int_{\mb{R}^n}\big(\eta u(t,x)\ml{A}^{\frac{2}{3}}\psi_{R}(t,x)-3u_t(t,x)\psi_{R,t}(t,x)-2\eta u(t,x)\ml{A}^{\frac{1}{3}}\psi_{R,t}(t,x)\big)\mathrm{d}x\Big|_{t=0}^{t=T_{\epsilon,\mathrm{w}}}.
\end{align*}
Without loss of generality, we assume the lifespan fulfilling
\begin{align*}
\sqrt{T_{\epsilon,\mathrm{w}}}>R_0:=\sqrt{2}r_0^{\frac{1}{3}\sigma}
\end{align*} so that concerning $R\in[R_0^2,T_{\epsilon,\mathrm{w}})$ we notice
\begin{align*}
\frac{T_{\epsilon,\mathrm{w}}+|x|^{\frac{2}{3}\sigma}}{R}\geqslant 1\ \ \mbox{for}\ \ x\in\mb{R}^n,\ \ \mbox{and}\ \ \frac{|x|^{\frac{2}{3}\sigma}}{R}\leqslant\frac{r_0^{\frac{2}{3}\sigma}}{R_0^2}=\frac{1}{2}\ \ \mbox{for}\ \ x\in B_{r_0}.
\end{align*}
For these settings, we know $\partial_t^k\psi_R(T_{\epsilon,\mathrm{w}},x)=0$ with $k=0,1,2$ for $x\in\mb{R}^n$, moreover,  $\psi_R(0,x)=1$ and $\partial_t^k\psi_R(0,x)=0$ with $k=1,2$ for $x\in B_{r_0}$. Benefited from the support conditions for initial data in a ball $B_{r_0}$, it leads to
\begin{align}\label{E7} 
&\epsilon\int_{\mb{R}^n}u_2(x)\mathrm{d}x+\int_0^{T_{\epsilon,\mathrm{w}}}\int_{\mb{R}^n}|u(t,x)|^p\psi_R(t,x)\mathrm{d}x\mathrm{d}t\notag\\
&\qquad\lesssim R^{-3}\int_0^{T_{\epsilon,\mathrm{w}}}\int_{\mb{R}^n}|u(t,x)|[\psi_R^*(t,x)]^{1-\frac{2\sigma}{m}}\mathrm{d}x\mathrm{d}t\notag\\
&\qquad\lesssim R^{(1+\frac{3n}{2\sigma})\frac{1}{p'}-3}\left(\int_0^{T_{\epsilon,\mathrm{w}}}\int_{\mb{R}^n}|u(t,x)|^p[\psi_R^*(t,x)]^{p-\frac{2\sigma p}{m}}\mathrm{d}x\mathrm{d}t\right)^{\frac{1}{p}},
\end{align}
in which we used H\"older's inequality associated with 
\begin{align*}
\mathrm{supp}\,\psi^*_R(t,x)\subset\left([0,R]\times B_{R^{\frac{3}{2\sigma}}}\right)\backslash\left\{(t,x):\ t+|x|^{\frac{2}{3}\sigma}\leqslant\tfrac{R}{2} \right\}.
\end{align*}

Let us now introduce the auxiliary functional
\begin{align*}
Y_p(R):=\int_0^Ry_p(r)r^{-1}\mathrm{d}r\ \ \mbox{with}\ \ y_p(r):=\int_0^{T_{\epsilon,\mathrm{w}}}\int_{\mb{R}^n}|u(t,x)|^p[\psi_r^*(t,x)]^{p-\frac{2\sigma p}{m}}\mathrm{d}x\mathrm{d}t.
\end{align*}
By letting $\tau=\frac{1}{r}(t+|x|^{\frac{2}{3}\sigma})$ as a new integral variable, according to the support condition and non-increasing property of $\chi(\tau)$, we are able to derive
\begin{align*}
Y_p(R)&=\int_0^R\left(\int_0^{T_{\epsilon,\mathrm{w}}}\int_{\mb{R}^n}|u(t,x)|^p[\psi_r^*(t,x)]^{p-\frac{2\sigma p}{m}}\mathrm{d}x\mathrm{d}t\right)r^{-1}\mathrm{d}r\\
&\leqslant\int_0^{T_{\epsilon,\mathrm{w}}}\int_{\mb{R}^n}|u(t,x)|^p\left(\int_{(t+|x|^{\frac{2}{3}\sigma})/R}^{\infty}\,[\chi^*(\tau)]^{(m-2\sigma)p}\tau^{-1}\mathrm{d}\tau\right)\mathrm{d}x\mathrm{d}t\\
&\leqslant\int_0^{T_{\epsilon,\mathrm{w}}}\int_{\mb{R}^n}|u(t,x)|^p\sup\limits_{r\in(0,R)}\left[\chi^*\left(\tfrac{t+|x|^{\frac{2}{3}\sigma}}{r}\right)\right]^{(m-2\sigma)p}\int_{1/2}^{1}\tau^{-1}\mathrm{d}\tau\mathrm{d}x\mathrm{d}t\\
&\leqslant\ln 2\int_0^{T_{\epsilon,\mathrm{w}}}\int_{\mb{R}^n}|u(t,x)|^p\left[\chi^*\left(\tfrac{t+|x|^{\frac{2}{3}\sigma}}{R}\right)\right]^{(m-2\sigma)p}\mathrm{d}x\mathrm{d}t
\end{align*}
since $m>2\sigma$. Using $\chi^*(\tau)\leqslant \chi(\tau)$, it holds
\begin{align}\label{E8}
Y_p(R)&\lesssim\int_0^{T_{\epsilon,\mathrm{w}}}\int_{\mb{R}^n}|u(t,x)|^p[\psi_R(t,x)]^{p-\frac{2\sigma p}{m}}\mathrm{d}x\mathrm{d}t=\int_0^{T_{\epsilon,\mathrm{w}}}\int_{\mb{R}^n}|u(t,x)|^p\psi_R(t,x)\mathrm{d}x\mathrm{d}t
\end{align}
by choosing $m>\max\{2\sigma,n+\frac{2}{3}\sigma \}$
 because of the critical case $p=p_{\mathrm{crit}}(n,\sigma)$ in our consideration. Recalling $Y'_p(R)=y_p(R)R^{-1}$, and combining with \eqref{E7} as well as \eqref{E8} so far we arrive at
\begin{align}\label{E9}
\left(\epsilon\int_{\mb{R}^n}u_2(x)\mathrm{d}x+Y_p(R)\right)^p\lesssim y_p(R)=R\,Y'_p(R)
\end{align}
for $R\in[R_0^2,T_{\epsilon,\mathrm{w}})$. Indeed, taking an integration over $[R_0,R]$ in \eqref{E9}, we have
\begin{align}\label{G1}
Y_p(R)\gtrsim\epsilon^{p_{\mathrm{crit}}(n,\sigma)}\left|\int_{\mb{R}^n}u_2(x)\mathrm{d}x\right|^p\int_{R_0}^R r^{-1}\mathrm{d}r\geqslant C_{u_2,p}\,\epsilon^{p_{\mathrm{crit}}(n,\sigma)}\ln R
\end{align}
under the sign condition \eqref{Sign_Assumption}. Another viewpoint of \eqref{E9} indicates
\begin{align*}
\frac{\mathrm{d}}{\mathrm{d}R}\left[\frac{1}{1-p}[Y_p(R)]^{1-p}\right]=[Y_p(R)]^{-p}\,Y'_p(R)\gtrsim R^{-1}.
\end{align*}
The integration of the last inequality over $[\sqrt{R},R]$ gives
\begin{align}\label{G2} 
\ln R\lesssim\frac{1}{p-1}\left([Y_p(\sqrt{R})]^{1-p}-[Y_p(R)]^{1-p}\right)\lesssim [Y_p(\sqrt{R})]^{1-p_{\mathrm{crit}}(n,\sigma)}.
\end{align}
Remember that $p_{\mathrm{crit}}(n,\sigma)>1$. We combine \eqref{G1} and \eqref{G2} to deduce
\begin{align*}
\ln R\lesssim C\epsilon^{p_{\mathrm{crit}}(n,\sigma)(1-p_{\mathrm{crit}}(n,\sigma))}(\ln R)^{1-p_{\mathrm{crit}}(n,\sigma)},
\end{align*}
as a consequence,
\begin{align*}
\ln\sqrt{T_{\epsilon,\mathrm{w}}}=\lim\limits_{R\uparrow\sqrt{T_{\epsilon,\mathrm{w}}}}\ln R\leqslant C\epsilon^{-(p_{\mathrm{crit}}(n,\sigma)-1)}.
\end{align*}
Finally, taking the action of the exponential function provides the completeness of our proof.

\section*{Acknowledgments}
The author was partially supported by the National Natural Science Foundation of China under Grant No. 12171317.

\end{document}